\documentclass[a4paper, 12pt, reqno]{amsart}

\usepackage{bm}
\usepackage{booktabs}
\usepackage[margin=2.5cm]{geometry}
\usepackage{stmaryrd} 
\usepackage{tabularx} 
\usepackage{tikz} 
\usepackage{tikz-cd}
\usepackage{hyperref}

\usepackage{adjustbox}

\allowdisplaybreaks 

\newcommand{\flplus} 
{
\mathbin{
\begin{tikzpicture}[baseline=-0.582ex]
    \draw [line width=0.24pt](-0.1129, 0) -- (0.1129, 0) -- (0, 0) -- (0, 0.1129) -- (0, -0.1129) arc (270:90:0.1129) -- (0, 0);
\end{tikzpicture}
}
}

\newcommand{\frplus} 
{
\hspace{0.1cm}
\begin{tikzpicture}[baseline=-0.582ex]
    \draw [line width=0.24pt](-0.1129, 0) -- (0.1129, 0) -- (0, 0) -- (0, 0.1129) -- (0, -0.1129) arc (-90:90:0.1129) -- (0, 0);
\end{tikzpicture}
\hspace{0.1cm}
}

\AtEndDocument{\vfill\eject\batchmode} 
\usepackage{amssymb}
\usepackage[T1]{fontenc}
\DeclareSymbolFont{script}{U}{eus}{m}{n}
\DeclareMathSymbol{\Wedge}{0}{script}{"5E}

\usepackage{t1enc}

\def\Cal{\mathcal}

\newcommand{\Rho}{\mathsf{P}}

\DeclareMathOperator{\G}   {G}
\DeclareMathOperator{\GL}  {GL}

\DeclareMathOperator{\Hom} {Hom}
\DeclareMathOperator{\id} {Id}
\DeclareMathOperator{\rank}{rank}

\DeclareMathOperator{\SL}  {SL}
\DeclareMathOperator{\SO}  {SO}
\DeclareMathOperator{\Sp}  {Sp}
\DeclareMathOperator{\SU}  {SU}

\DeclareMathOperator{\U}   {U}

\newcommand{\wt}[1]{\widetilde{#1}}

\newcommand{\bbC}{\mathbb{C}}

\newcommand{\bbH}{\mathbb{H}}
\newcommand{\bbI}{\mathbb{I}}
\newcommand{\bbJ}{\mathbb{J}}

\newcommand{\bbK}{\mathbb{K}}

\newcommand{\bbR}{\mathbb{R}}

\newcommand{\bbV}{\mathbb{V}}

\newcommand{\bbZ}{\mathbb{Z}}

\newcommand{\wh} [1]{\smash{\widehat {#1}}}

\newcommand{\mbc}{\mathbf{c}}
\newcommand{\mbg}{\mathbf{g}}

\newcommand{\mbp}{\mathbf{p}}
\newcommand{\mbQ}{\mathbf{Q}}

\newcommand{\mcA}{\mathcal{A}}

\newcommand{\mcE}{\mathcal{E}}

\newcommand{\mcL}{\mathcal{L}}

\newcommand{\mcS}{\mathcal{S}}
\newcommand{\mcT}{\mathcal{T}}

\newcommand{\mfg}  {\mathfrak{g}}

\newcommand{\wtM}{\wt M}
\newcommand{\wtP}{\wt P}
\newcommand{\wtmcT}{\wt\mcT}

\def\R{\mathbb{R}}

\newcommand{\ce}{{\Cal E}}

\newcommand{\cT}{{\mathcal T}}

\newcommand{\rpl}                         
{\mbox{$
\begin{picture}(12.7,8)(-.5,-1)
\put(0,0.2){$+$}
\put(4.2,2.8){\oval(8,8)[r]}
\end{picture}$}}

\def\R{\mathbb{R}}

\newcommand{\si}{\sigma}

\newcommand{\Aut}{\operatorname{Aut}}
\newcommand{\End}{\operatorname{End}}

\newtheorem{theorem}{Theorem}[section]
\newtheorem{lemma}[theorem]{Lemma}
\newtheorem{proposition}[theorem]{Proposition}

\theoremstyle{definition}
\newtheorem{definition}[theorem]{Definition}

\newtheorem{notation}[theorem]{Notation}
\theoremstyle{remark}
\newtheorem{remark}[theorem]{\rm\bf Remark}
\newtheorem*{definition*}{\rm\bf Definition}

\newcommand{\Ric}{\operatorname{Ric}}

\newcommand{\vol}{\bm \epsilon}

\def\sideremark#1{\ifvmode\leavevmode\fi\vadjust{\vbox to0pt{\vss
 \hbox to 0pt{\hskip\hsize\hskip1em
 \vbox{\hsize3cm\tiny\raggedright\pretolerance10000
  \noindent #1\hfill}\hss}\vbox to8pt{\vfil}\vss}}}%

\author{A.\ R. Gover, K. Neusser, \& T.\ Willse}

\title{Projective geometry of $3$-Sasaki structures}
\begin{document}

\address{
A.R.G.: Department of Mathematics\\
  The University of Auckland\\
  Private Bag 92019\\
  Auckland 1142\\
  New Zealand\\
K.N.: Department of Mathematics and Statistics\\
  Masaryk University\\
  Kotl\'a\v rsk\'a 267/2 \\
  611 37 Brno\\
  Czech Republic\\
T.W.: Mathematics Department\\
  Guilford College\\
  5800 W Friendly Ave\\
  Greensboro NC 27410\\
  United States of America
}
\email{r.gover@auckland.ac.nz}
\email{neusser@math.muni.cz}
\email{twillse@guilford.edu}

\subjclass[2010]{}
\begin{abstract}
We show that $3$-Sasaki structures admit a natural description in
terms of projective differential geometry. This description provides a
concrete link between $3$-Sasaki structures and several other
geometries and constructions via a single unifying picture.  First we
establish that a $3$-Sasaki structure may be understood as a
projective structure equipped with a certain holonomy reduction to the
(possibly indefinite) unitary quaternionic group $\Sp(p,q)$, namely a parallel
hyperk\"ahler structure on the projective tractor bundle satisfying a
particular genericity condition. For the converse, where one begins
with a general parallel hyperk\"ahler structure on the projective
tractor bundle, the genericity condition is not automatic. Indeed we
prove that generically such a reduction decomposes the underlying
manifold into a disjoint union of strata including open manifolds
with (indefinite) $3$-Sasaki structures and a closed separating
hypersurface at infinity with respect to the
$3$-Sasaki metrics. Moreover, it is shown that the latter hypersurface
inherits a Biquard--Fefferman conformal structure, which thus
(locally) fibres over a quaternionic contact structure, and which in turn
compactifies the natural quaternionic K\"ahler quotients of the
$3$-Sasaki structures on the open manifolds.  As an application we
describe the projective compactification of (suitably) complete,
non-compact (indefinite) $3$-Sasaki manifolds and recover Biquard's
notion of asymptotically hyperbolic quaternionic K\"ahler metrics.
\end{abstract}

\maketitle

\section{Introduction}
A (pseudo-)Riemannian manifold $(M, g)$ is said to be Sasaki, or
Sasakian if its standard metric cone is
(pseudo-)K\"{a}hler \cite{BG-book}, and said to be Sasaki--Einstein if $g$ is also
Einstein. Inspired by the pioneering work of
\cite{ArmstrongThesis,ArmstrongP1,ArmstrongP2}, we showed in \cite{GNW} that
projective geometry provides a natural unifying framework for treating
many aspects of Sasaki and Sasaki--Einstein geometry. Recall that a
projective structure on a manifold is an equivalence class $\mbp$
of torsion-free affine connections that share the same unparametrised
geodesics. Note that a (pseudo-)Riemannian manifold---and in particular a Sasaki manifold---thus induces a
projective structure by its Levi-Civita connection. As shown in \cite{GNW}, a gain of this perspective
is that it leads to natural geometric compactifications of some complete
non-compact (indefinite) Sasaki--Einstein structures, using tools from \cite{CGcompact,CGH}.

Canonically associated to any projective manifold $(M,\mbp)$ of
dimension $n+1\geq2$ is a projectively invariant connection
$\nabla^\cT$ on a natural vector bundle $\cT$ of rank $n+2$, the
so-called tractor bundle with connection \cite{BEG}.  A result, with
close links to the metric cone characterisation of Sasaki-structures,
is that a Sasaki--Einstein $(2m+1)$-manifold of positive definite
signature is exactly the same as a projective manifold for which the
restricted holonomy of $\nabla^\cT$ is contained in $\SU(m+1)$. In the
case of holonomy reductions of $\nabla^\mcT$ to the indefinite special
unitary group $\SU(p,q)$ the situation is more subtle and interesting,
and these subtleties, not visible in the metric cone picture, are what
is studied and exposed in \cite{GNW}. Specifically, we showed in
\cite[Theorem B]{GNW} that a projective holonomy reduction to
$\SU(p,q)$ induces in general a stratification
\[
    M = M_+ \cup M_0 \cup M_-
\]
of the underlying projective manifold $(M, \mbp)$ into a disjoint union of submanifolds. Here,
$M_\pm$ are open submanifolds equipped with indefinite Sasaki--Einstein
structures, and $M_0$ is a smooth separating hypersurface equipped
with an oriented conformal structure of signature $(2p - 1, 2q - 1)$
whose {\em conformal holonomy} is contained in $\SU(p,q)$, making it
(locally) what is called a Fefferman space
(cf.\ \cite{CGFeffchar,Leit}). In fact $M_0$ is the compactifying
projective infinity of the Sasaki--Einstein spaces $M_\pm$ \cite[Theorem C]{GNW}. To
understand this picture fully, we explored in \cite{GNW} the group identity
\begin{equation}\label{keydis}
\U(p, q) = \SO(2p, 2q) \cap \Sp(2m + 2, \bbR)\cap \GL(m + 1, \bbC),
\end{equation}
and the fact that $\U(p, q)$ is the (compatible) intersection of any
two of the groups on the right-hand side, at the level of projective
holonomy reductions.  We there treat the holonomy reductions for each
group on the right-hand side separately, and then the implications of
pairs of them holding simultaneously in a compatible way.

In this article we develop the analogous theory for {\em $3$-Sasaki
  structures}, meaning those (pseudo-)Riemannian manifolds $(M,g)$
whose standard metric cone is hyperk\"ahler, which implies in
particular that $\dim M \equiv 3 \pmod 4$.  In recent decades
$3$-Sasaki geometries have been the subject of considerable focus for
a host of different reasons \cite{AF,BG-3Sas,BG-book,BGM,GS,GWZ}. From
our current point of view they are extremely interesting because, via
the tools of projective geometry and the curved orbit theory of
\cite{CGH}, we are able to link indefinite $3$-Sasaki structures to
Biquard--Fefferman conformal structures, which are conformal
structures that fibre (locally) over quaternionic contact structures via a
construction described by Biquard in \cite{BiquardMetriques}; for a
holonomy characterisation of these conformal structures see
\cite{AltQuaternionic}. Moreover, quaternionic contact structures were
introduced by Biquard in \cite{BiquardMetriques, BiquardQuaternionic}
as geometric structures arising at the infinity of quaternionic K\"ahler
manifolds, and our study recovers that fundamental
relationship. Indeed all the geometries just mentioned are shown to be
parts of a single smooth geometric structure and its leaf spaces. In
more detail, the article proceeds as follows.

Following a recollection of both the algebraic preliminaries and the
notions of quaternionic and hyperk\"{a}hler structures in Section
\ref{alg}, in Definition \ref{Def_3Sas} we recall the usual definition
of a $3$-Sasaki structure in terms of a triple of Killing vector fields. Proposition \ref{coneprop} then
relates this description to the aforementioned definition via the metric
cone. In Section \ref{subsection:affine-projective-geometry} we review
the necessary background on affine and projective differential
geometry, including the standard tractor bundle $(\mcT, \nabla^\mcT)$ of a projective manifold.

The main results are presented in Section \ref{mainsect}. In Theorem
\ref{thmA} we show that the projective tractor bundle $(\mcT,\nabla^\mcT)$ of
a $3$-Sasaki manifold of dimension $4m+3$ and of arbitrary signature $(4p-1,4q)$ admits
a parallel tractor hyperk\"ahler structure of signature $(4p,4q)$,
equivalently a holonomy reduction to
\begin{equation}\label{group_intersection}
    \Sp(p, q) = \Sp(2 m + 2, \bbC) \cap \U(2 p, 2 q) = \SO(4 p, 4 q) \cap \GL(m + 1, \bbH) ,
\end{equation}
where $m := p + q - 1$. Theorem \ref{thmB} then considers the
converse, namely a projective manifold $(M,\mbp)$ of dimension
$4m+3\geq 11$ equipped with a parallel hyperk\"ahler structure of
signature $(4p,4q)$ on its tractor bundle. We show that, in general, this determines a stratification  of $M$:
$$
M = M_+ \cup M_0 \cup M_-,
$$
where $M_\pm$
are open submanifolds equipped with indefinite $3$-Sasaki structures
of signature $(4p-1,4q)$ and $(4p,4q-1)$ respectively, and $M_0$ is a
smooth separating hypersurface equipped with a Biquard--Fefferman
conformal structure of signature $(4p-1, 4q-1)$, as mentioned above and
whose definition is recalled in Section \ref{sec_q_c}.
Theorem \ref{thmC}
then observes that $M_0$ with its conformal structure is the boundary (``at infinity,'' with respect to the $3$-Sasaki metric) of
a projective compactification of the $3$-Sasaki structures on
$M_\pm$; see \cite{CGcompact,CGHjlms} for the notion of projective
compactification.  In Section \ref{Sec_Hol_SL} we study oriented projective manifolds $(M,\mbp)$ of dimension $4m+3\geq 11$ with a parallel hypercomplex structure on its tractor bundle, that is, with a projective holonomy reduction
to
\[
    \SL(m + 1, \bbH) = \GL(m + 1, \bbH) \cap \SL(4 m + 4, \bbR).
\]
We show that such a reduction determines an integrable distribution $D$ of rank $3$ on $M$ and that (locally) the
corresponding leaf space $\widetilde M$ inherits a quaternionic structure $\widetilde Q$; see Theorem \ref{Thm_quaternionic_descent}. Moreover, in Theorem \ref{thm:tractor_descent} we prove that in this setting
the projective tractor bundle of $(M,\mbp)$ descends, via the natural projection $\pi: M\rightarrow \widetilde M$, to the canonical (normal) quaternionic tractor bundle of $(\widetilde M, \widetilde Q)$---a useful result of independent interest.
Finally, that fact  is used to show in Theorem \ref{thmD} that, if the projective holonomy reduces from $\SL(m + 1, \bbH) $ further to $\Sp(p,q)$, then
$\widetilde M$ admits in general a decomposition into submanifolds
\[
    \wt{M} = \wt {M}_+ \cup \wt{M}_0 \cup \wt{M}_-,
\]
where $\wt{M}_{\pm}$ are open submanifolds with quaternionic K\"ahler structures of signatures $(4(p-1), 4q)$ and  $(4(q-1), 4p)$ respectively,
and $\wt{M}_0$ is a real separating hypersurface with a quaternionic contact structure of signature $(p-1, q-1)$ (see Definition \ref{quat_contact}). The quaternionic K\"ahler structures on $\wt{M}_{\pm}$
coincide here with the canonical quaternionic K\"ahler  quotients of the $3$-Sasaki structures on
$M_\pm$ (induced  by the defining triple of Killing fields), and the quaternionic contact structure on $\wt{M}_0$ coincides with the one induced from the Biquard--Fefferman conformal structure on  $M_0$.

\smallskip

As a final point here we remind the reader that the metric cone over a
Sasaki--Einstein manifold carries a Ricci-flat K\"ahler structure. In
dimension $4$, however, a manifold carries such a structure if and
only if it can be extended to a hyperk\"ahler structure. Thus, a
$3$-dimensional Sasaki--Einstein structure can be extended to a
$3$-Sasaki structure. In that sense, the case of $3$-dimensional
$3$-Sasaki structures was essentially treated in \cite[\S1]{GNW}. In
any case, as pointed out there, the metric underlying any Sasaki
structure (and hence any $3$-Sasaki structure) in that dimension is
projectively flat. We will focus in this article on $3$-Sasaki
structures of dimension at least $11$ and will discuss the
$7$-dimensional case, which is also special, only in remarks.

All objects herein are smooth, meaning $C^\infty$.

\medskip

\noindent \textbf{Acknowledgements:}
A.R.G.\ acknowledges support by the Royal Society of New Zealand via Marsden Grants 16-UOA-051 and 19-UOA-008. K.N.\ was
supported by the grant GA19-06357S from Czech Science Foundation (GA\v CR). T.W. acknowledges support and hospitality from Centro de Investigaci\'on en Matem\'aticas (Mexico), and support from Guilford College.

\section{Background}
We start by recalling the notions of various geometric structures and some basic results about them that will be central for this article.

\subsection{Quaternionic algebraic and geometric structures} \label{alg}
For a uniform treatment of the quaternionically ``flavoured'' structures in this section, see \cite{AlekseevskyMarchiafava}.

\subsubsection{Quaternionic and hypercomplex vector spaces}

\begin{definition} Suppose $\bbV$ is a real (finite-dimensional) vector space.
\begin{itemize}
\item[(a)] A \textit{hypercomplex structure} on  $\bbV$ is a triple $(\bbI, \bbJ, \bbK)$ of complex structures
$\bbI,\bbJ,\bbK \in \End \bbV$ on $\bbV$ satisfying
\[
    \bbI \bbJ \bbK = -\smash{\id_\bbV} .
\]
\item[(b)] A \textit{quaternionic structure} on $\bbV$ is a $3$-dimensional subspace
\[
    \mbQ := \operatorname{span} \{\bbI, \bbJ, \bbK\} \subseteq \End \bbV,
\]
where $(\bbI, \bbJ, \bbK)$ is a hypercomplex structure
\end{itemize}
\end{definition}
Note that, if $\mbQ\subseteq \End \bbV$ is a quaternionic structure on
a vector space $\bbV$ and $(\bbI, \bbJ, \bbK)$ a hypercomplex
structure inducing it, then
declaring $(\bbI, \bbJ, \bbK)$ to be an oriented, orthonormal basis
determines an orientation and inner product on $\mbQ$.  Moreover, the
hypercomplex structures that induce $\mbQ$ are exactly the oriented
bases of $\mbQ$ that are orthonormal with respect to that inner product.

If $\bbV$ admits a hypercomplex or quaternionic structure, then $\dim_\bbR \!\bbV = 4m + 4$ for some $m \in \{0, 1, \ldots\}$.
As for complex structures, a hypercomplex or quaternionic structure determines an orientation on $\bbV$: Given a quaternionic basis of $\bbV$, that is, an ordered $(m + 1)$-tuple $(E_1, \ldots, E_{m+1})$ of vectors in $\bbV$ such that
\[
    (E_1, \bbI E_1, \bbJ E_1, \bbK E_1, \ldots, E_{m + 1}, \bbI E_{m + 1}, \bbJ E_{m + 1}, \bbK E_{m+1})
\]
is a basis of $\bbV$, declaring
\begin{equation}\label{orientation_q}
    E_1 \wedge \bbI E_1 \wedge \bbJ E_1 \wedge \bbK E_1 \wedge \cdots \wedge E_{m + 1} \wedge \bbI E_{m + 1} \wedge \bbJ E_{m + 1} \wedge \bbK E_{m + 1}
        \in \Wedge^{4 m + 4} \bbV
\end{equation}
to be positively oriented defines an orientation on $\bbV$. This
orientation is independent of the choice of vectors $E_1, \ldots,
E_{m+1}$ and the choice of representative triple $(\bbI,\bbJ,\bbK)$ in
the quaternionic case, since any two such representatives are related
by an action of $\SO(3)$.

If $\bbV$ is equipped with a hypercomplex structure $(\bbI, \bbJ,
\bbK)$ then, regarding it as a (right) $\bbH$-module---where $\bbH$ is
the quaternion algebra generated by $(\bbI, \bbJ, \bbK$)---its
automorphism group is
\[
    \Aut(\bbV, (\bbI, \bbJ, \bbK)) \cong \GL(m + 1, \bbH) ,
\]
where $\dim_\bbR \!\bbV = 4 m + 4$. The automorphism group of the
induced quaternionic structure $\mbQ=\operatorname{span} \{\bbI, \bbJ,
\bbK\}$ on $\bbV$ is
\[
    \Aut(\bbV, \mbQ) \cong \GL(m + 1, \bbH)  \times_{\bbZ_2} \Sp(1) := (\GL(m + 1, \bbH) \times \Sp(1)) / \bbZ_2.
\]
The factor $\Sp(1)=\{q\in \bbH: |q|=1\}$ acts on $\bbV$ by
quaternionic scalar multiplication on the right, and $\bbZ_2$ acts by
multiplication by $\pm\smash{\id}$ on both factors.

The maximal subgroups of these two groups characterised by fixing a volume form
$\vol \in \Wedge^{4 m + 4} \bbV^*$ compatible with the canonical
orientation on $\bbV$ are
\begin{align*}
    \Aut(\bbV, (\bbI, \bbJ, \bbK), \vol)
        &\cong \SL(m + 1, \bbH) := \GL(m + 1, \bbH) \cap \SL(4 m + 4, \bbR) , \\
    \Aut(\bbV, \mbQ, \vol)
        &\cong \SL(m + 1, \bbH)  \times_{\bbZ_2} \Sp(1).
\end{align*}
\begin{definition}
Let $\bbV$ be a real vector space equipped with an inner product $h$, that is, a symmetric non-degenerate bilinear form (here, and henceforth, not necessarily positive definite).
\begin{itemize}
\item[(a)] The inner product
$h$ is called \textit{Hermitian}, with respect to a hypercomplex structure $(\bbI, \bbJ, \bbK)$ on $\bbV$, if
\begin{equation}\label{Hermitian}
    h(X, Y)
        = h(\bbI X, \bbI Y)
        = h(\bbJ X, \bbJ Y)
        = h(\bbK X, \bbK Y)
    \qquad
    \textrm{for all}
    \qquad
    X, Y \in \bbV,
\end{equation} or, with respect to a quaternionic structure $\mbQ$ on $\bbV$, if \eqref{Hermitian} holds for any, equivalently, every, hypercomplex structure inducing $\mbQ$.
\item[(b)] If an inner product $h$ is Hermitian with respect to a hypercomplex structure $(\bbI, \bbJ, \bbK)$ or a quaternionic structure $\mbQ$, we call the respective pairs $(h, (\bbI, \bbJ, \bbK))$ and
$(h, \mbQ)$ a \textit{Hermitian hypercomplex structure} and a \textit{Hermitian quaternionic structure} on $\bbV$.
\end{itemize}
\end{definition}

If an inner product $h$ on a real vector space $\bbV$ is Hermitian with respect to a hypercomplex or quaternionic structure thereon, then its signature is $(4p, 4q)$ for some $p, q \in \{0, 1, \ldots\}$. Since the hypercomplex or quaternionic structure determines a natural orientation on $\bbV$, a Hermitian inner product $h$ determines a preferred volume form on $\bbV$.
Regarded as a (right) $\bbH$-module, a vector space $\bbV$ equipped with a Hermitian hypercomplex structure $(h, (\bbI, \bbJ, \bbK))$ has automorphism group
\[
    \Aut(\bbV, h, (\bbI, \bbJ, \bbK)) \cong \Sp(p, q) ,
\]
in the case where the signature of $h$ is $(4p, 4q)$, and the automorphism group of the induced Hermitian quaternionic structure $(h, \mbQ)$ on $\bbV$ is
\[
\Aut(\bbV, h, \mbQ) \cong \Sp(p, q) \times_{\bbZ_2} \Sp(1).
\]
Figure 1 adapts from \cite[Section 1]{AlekseevskyMarchiafava} a diagram of the relationships of the above quaternionically flavoured structures on a real vector space,
giving for each the name, the defining data and its automorphism group. Arrows denote inclusions of automorphism groups. i.e. forgetting some of the data defining the structure.
\noindent
\begin{figure}[h]
       \adjustbox{width=\textwidth}{
    \begin{tikzcd}[column sep=large, row sep=large]
        \begin{tabular}{c}Hermitian hypercomplex\\$(h, (\bbI, \bbJ, \bbK))$\\$\Sp(p, q)$\end{tabular} \arrow[r, hook] \arrow[d, hook]
        &
        \begin{tabular}{c}unimodular hypercomplex\\$((\bbI, \bbJ, \bbK), \vol)$\\$\SL(m + 1, \bbH)$\end{tabular} \arrow[r, hook] \arrow[d, hook]
        &
        \begin{tabular}{c}hypercomplex\\$(\bbI, \bbJ, \bbK)$\\$\GL(m + 1, \bbH)$\end{tabular} \arrow[d, hook]
        \\
        \begin{tabular}{c}Hermitian quaternionic\\$(h, \mbQ)$\\$\Sp(p, q) \times_{\bbZ_2} \Sp(1)$\end{tabular} \arrow[r, hook]
        &
        \begin{tabular}{c}unimodular quaternionic\\$(\mbQ, \vol)$\\$\SL(m + 1, \bbH) \times_{\bbZ_2} \Sp(1)$\end{tabular} \arrow[r, hook]
        &
        \begin{tabular}{c}quaternionic\\$\mbQ$\\$\GL(m + 1, \bbH) \times_{\bbZ_2} \Sp(1)$\end{tabular}
        \\
    \end{tikzcd}
    }
     \caption{Relationships among quaternionically flavoured structures.}
\end{figure}
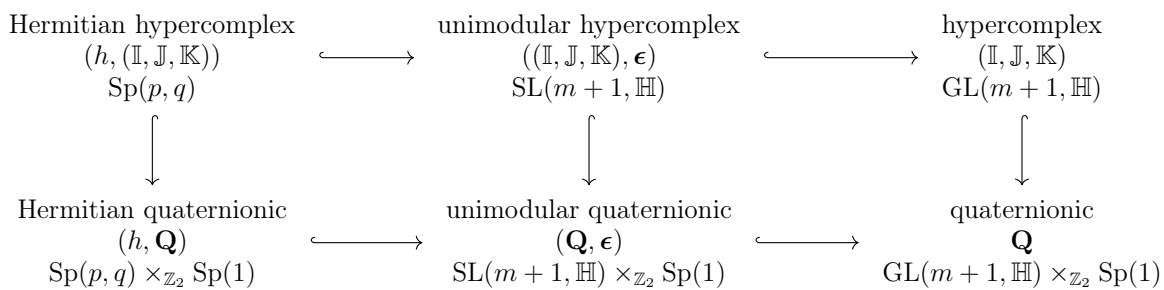

\subsubsection{Quaternionic and quaternionic K\"ahler structures}

We now transfer the algebraic notions from Section \ref{alg} to the vector bundle setting.

\begin{definition}\label{almost_quat_str}
~
\begin{itemize}
\item[(a)]
  A \emph{quaternionic structure} on a vector bundle $\mathcal{V}$ over a manifold $M$ is a rank-$3$ subbundle $Q_{\mathcal{V}} \subset \End(\mathcal{V})$ such that, locally around any point in $M$, there exists a local frame
$(I, J, K)$ of $Q_{\mathcal{V}}$ satisfying the quaternionic relations:
\[
    I^2=J^2=K^2=-\smash{\id_{\mathcal V}}
    \qquad \textrm{and} \qquad
    IJK=-\smash{\id_{\mathcal V}} ,
\]
where $\id_{\mathcal V}$ is the identity endomorphism. Such local frames $(I,J,K)$ of a quaternionic structure are called \emph{admissible}. It follows that $\rank \mathcal{V} \equiv 0 \pmod 4$.
\item[(b)] An \emph{almost quaternionic structure} on a manifold $M$ of dimension $4m+4$ is a quaternionic structure $Q$ on $TM$, equivalently, a first-order $\G$-structure with structure group $\GL(m + 1,\bbH) \times_{\bbZ_2} \Sp(1)$.
    \end{itemize}
\end{definition}

Note that an almost quaternionic manifold $(M,Q)$ admits a canonical orientation, namely that determined by any admissible local frame $(I,J,K)$ of $Q$ as in \eqref{orientation_q}.
Moreover, it is well-known, see Section 3.2 of \cite{AlekseevskyMarchiafava}, that an almost quaternionic manifold $(M,Q)$ admits a distinguished class of linear connections on
$TM$ that preserve $Q$ and have the same minimal torsion, which may be identified with the intrinsic torsion of the $\GL(m + 1,\bbH) \times_{\bbZ_2} \Sp(1)$-structure.
This class of connections forms an affine space over the vector space of $1$-forms on $M$. Specifically, any two such connections, $\nabla$ and $\widehat{\nabla}$,
are related by a $1$-form $\Upsilon\in\Gamma(T^*M)$ according to the formula
\begin{equation}\label{quater_proj_change}
	\widehat{\nabla}_{\eta} \xi
		= \nabla_{\eta} \xi
			+ \Upsilon(\eta)\xi + \Upsilon(\xi)\eta
			- \sum_{A \in \{I, J, K\}} [(\Upsilon(A\eta) A\xi + \Upsilon(A\xi) A\eta)] ,
\end{equation}
where $(I, J, K)$ is an admissible local frame of $Q$.

\begin{definition}\label{quaternionic_st}
An almost quaternionic manifold $(M, Q)$ of dimension $4m+4 \geq 8$ is called
\emph{quaternionic} if it admits a torsion-free linear connection on
$TM$ preserving $Q$, equivalently, if the intrinsic torsion of the corresponding $\GL(m + 1,\bbH) \times_{\bbZ_2} \Sp(1)$-structure vanishes.
\end{definition}

\begin{definition}\label{Def_QK}
  A \textit{quaternionic K\"ahler structure} on a manifold $M$ of dimension $4m+4 \geq 8$
consists of a (pseudo-)Riemannian metric $g \in \Gamma(S^2 T^*M)$ and an almost
quaternionic structure $Q \subset \End(TM)$ such that
\begin{itemize}
\item for each $x \in M$, $(g_x, Q_x)$ is a Hermitian quaternionic structure on $T_x M$;
\item $Q$ is parallel with respect to the Levi-Civita connection of $g$.
\end{itemize}
In particular, $Q$ is quaternionic.
\end{definition}

A connected (pseudo-)Riemannian manifold $(M, g)$ of signature $(4p, 4q)$ admits a quaternionic K\"ahler structure if and only if the holonomy of $g$ is contained in $\Sp(p, q) \times_{\bbZ_2} \Sp(1)$. Quaternionic K\"ahler manifolds are known to be Einstein; see e.g.\,Corollary 12.2.14 in \cite{BG-book}.

\begin{remark}\label{QK_dim4}
Note that in Definitions \ref{quaternionic_st} and \ref{Def_QK} we required the dimension to be at least $8$. The reason is that dimension $4$ is special and requires a separate treatment.
Let us briefly explain this: Almost quaternionic structures in dimension $4$ are known to be automatically quaternionic in the sense of Definition \ref{quaternionic_st}. Moreover, in dimension $4$,
due to the isomorphism $\Sp(1) \times_{\bbZ_2} \Sp(1)\cong\textrm{SO}(4)$ (and thus also $\GL(1, \mathbb H) \times_{\bbZ_2} \Sp(1)\cong\textrm{CSO}(4)$), almost quaternionic structures are the same as oriented conformal structures,
and Riemannian manifolds with holonomy
in $\Sp(1) \times_{\bbZ_2} \Sp(1)$ are simply oriented Riemannian manifolds. In view of that, the definition of quaternionic K\"ahler $4$-manifolds needs to be adjusted: they are defined as oriented
Riemannian $4$-manifolds that are self-dual and Einstein. Similarly, the definition of quaternionic structures in dimension $4$ is sometimes adjusted to mean self-dual oriented conformal structures;
see e.g.\,Definition 12.2.12 in \cite{BG-book}.
\end{remark}

\subsubsection{Hypercomplex and hyperk\"ahler structures}  \label{hyp-str-sect}

\begin{definition}
An \emph{(almost) hypercomplex structure} on a manifold $M$ consists of
a triple $(I, J, K)$ of (almost) complex structures satisfying the
quaternionic relation ${IJK = -\smash{\id_{TM}}}$.
\end{definition}

As for almost quaternionic structures, almost hypercomplex manifolds admit a natural orientation, which coincides with
the orientation induced by any of the three almost complex structures.

\begin{definition}
A \textit{hyperk\"ahler structure} on a manifold $M$ consists of a
(pseudo-)Riemannian metric $g \in \Gamma(S^2 T^*M)$ and an almost
hypercomplex structure $(I, J, K)$ such that
\begin{itemize}
\item for each $x \in M$, $(g_x, (I_x, J_x, K_x))$ is a Hermitian hypercomplex structure on $T_x M$, and
\item $I$, $J$, and $K$ are each parallel with respect to the Levi-Civita connection $\nabla$ of $g$.
  \end{itemize}
In particular, since $I$, $J$ and $K$ are preserved by a torsion-free connection, they are necessarily integrable, i.e.\,$(I,J,K)$ is a hypercomplex structure.
\end{definition}

Put another way, a \textit{hyperk\"ahler structure} on a manifold $M$
is a quaternionic K\"ahler structure $(g, Q)$ on $M$ together with
a global frame of $Q$ that is parallel with respect to the
Levi-Civita connection $\nabla$ of $g$. Put yet another way, a
hyperk\"ahler structure on a manifold $M$ consists of a metric $g \in
\Gamma(S^2 T^*M)$ together with a triple $(\bbI, \bbJ, \bbK)$ of
complex structures satisfying the quaternionic relations such that
$(g, \bbI)$, $(g, \bbJ)$, $(g, \bbK)$ are K\"ahler structures.

A connected (pseudo-)Riemannian manifold $(M, g)$ of signature $(4 p,
4 q)$ admits a hyperk\"ahler structure if and only if the holonomy of
$\nabla$ is contained in $\Sp(p, q)$. Locally, hyperk\"ahler
structures are precisely the quaternionic K\"ahler structures that are
Ricci-flat; see e.g.\,Corollary 12.2.14 in
\cite{BG-book}.\footnote{Some sources define quaternionic K\"ahler
  structures to exclude hyperk\"ahler structures.}

\subsection{$3$-Sasaki structures} \label{3-S} We recall some background on Sasaki and $3$-Sasaki structures; for a comprehensive treatment we refer to the monograph \cite{BG-book}.

\begin{definition}\label{Def_Sasaki}
A \textit{Sasaki structure} on a manifold $M$ consists of a (pseudo-)Riemannian metric $g_{ab} \in \Gamma(S^2 T^*M)$ and a Killing field $k^a \in \Gamma(TM)$ of $g$ such that
\begin{enumerate}
	\item $g_{ab} k^a k^b = 1$
	\item $\nabla_a\nabla_b k^c = -g_{ab} k^c + \delta^c{}_a k_b .$
\end{enumerate}
\end{definition}

The following equivalent characterisation is well-known; see \cite{BG-3Sas, BG-book}.

\begin{proposition}\label{proposition:metric-cone-characterization}
Let $(M, g)$ be a (pseudo-)Riemannian manifold and
\[
    (\wh M, \hat g) := (M \times \bbR_+, dt^2 + t^2 g)
\]
its metric cone, where $t$ is the standard coordinate on $\bbR_+$ and we identify $M \leftrightarrow M \times \{1\}$.
\begin{enumerate}

\item If $k \in \Gamma(TM)$ is a vector field such that $(M, g, k)$ is
  a Sasaki manifold, define $\bbK \in \Gamma(\End(T\wh M))$ with
  respect to the splitting $T\wh M = T(M \times \bbR_+) \cong TM
  \oplus T\bbR_+$ by declaring
  \[
    \bbK(v) := \nabla_v k \quad \textrm{for } v \in \{k\}^{\perp} \cap TM, \qquad
    \bbK(k) := t \partial_t, \qquad \textrm{and} \qquad
    \bbK(t\partial_t) := -k .
  \]
  Then, $\bbK$ is a complex structure, and $(\wh M,
  \hat g, \bbK)$ is a K\"ahler manifold.

\item Conversely, if $\bbK \in \Gamma(\End(T\wh M))$ is a complex
  structure for which $(\wh M, \hat g, \bbK)$ is a K\"ahler structure,
  define $k \in \Gamma(TM)$ by $k := \bbK \partial_t \vert_M$. Then,
  $(M, g, k)$ is a Sasaki structure.
\end{enumerate}
\end{proposition}
\noindent In particular, any manifold admitting a Sasaki structure has odd dimension.

\medskip

\begin{definition}\label{Def_3Sas}
A \textit{$3$-Sasaki structure} on a manifold $M$ comprises a (pseudo-)Riemannian metric $g_{ab} \in \Gamma(S^2 T^*M)$ and a triple $(i, j, k)$ of Killing fields of $g$, where $i, j$, and $k$ separately satisfy the conditions $(a)$ and $(b)$ of Definition \ref{Def_Sasaki} and together satisfy the following ``quaternionic conditions'':
\begin{enumerate}
    \item[(c)] $g_{ab}i^aj^b=g_{ab}j^ak^b=g_{ab}k^ai^b=0$
    \item[(d)] $[i, j]=-2k$, \qquad $[j, k]=-2i$, \qquad and \qquad $[k,i]=-2j$.
\end{enumerate}
The collection $(M, g, (i, j, k))$ is together called a \textit{$3$-Sasaki manifold}.
\end{definition}

It follows immediately from the definitions that for any $3$-Sasaki manifold $(M, g, (i, j, k))$, each of $(M, g, i)$, $(M, g, j)$, and $(M, g, k)$ are Sasaki manifolds.

Proposition \ref{proposition:metric-cone-characterization} immediately implies an analogous result---a characterisation via the metric cone---for $3$-Sasaki manifolds vis-\`a-vis hyperk\"ahler manifolds; see \cite{BG-3Sas, BG-book}.

\begin{proposition} \label{coneprop}
Let $(M, g)$ be a (pseudo-)Riemannian manifold and
\[
    (\wh M, \hat g) := (M \times \bbR_+, dt^2 + t^2 g)
\]
its metric cone, where $t$ is the standard coordinate on $\bbR_+$, and identify $M \leftrightarrow M \times \{1\}$.
\begin{enumerate}
\item If $(M, g, (i, j, k))$ is a $3$-Sasaki manifold, respectively define $\bbI, \bbJ, \bbK \in \Gamma(T\wh M)$ as in Proposition \ref{proposition:metric-cone-characterization}(a). Then, $(\wh M, \hat g, (\bbI, \bbJ, \bbK))$ is a hyperk\"ahler manifold.
\item For the converse, if $(\wh M, \hat g, (\bbI, \bbJ, \bbK))$ is a hyperk\"ahler manifold, respectively define $i, j, k \in \Gamma(TM)$ as in Proposition \ref{proposition:metric-cone-characterization}(b). Then, $(M, g, (i, j, k))$ is a $3$-Sasaki structure.
\end{enumerate}
\end{proposition}
\noindent In particular, if a manifold $M$ admits a $3$-Sasaki structure, then $\dim M \equiv 3 \pmod 4$.

Given a $3$-Sasaki structure on a manifold $M$, the identification $M
\leftrightarrow M \times \{1\}$, the canonical orientation on its
hyperk\"ahler cone, and the preferred unit normal vector field
$\partial_t \vert_M$ along $M$, together determine a canonical orientation and
hence (via the metric) a preferred volume form on $M$.
Moreover, a metric cone of a (pseudo-)Riemannian manifold $(M, g)$ is Ricci-flat if and only $g$ is Einstein \cite[Lemma 11.1.15]{BG-book}. Thus Ricci-flatness of hyperk\"ahler manifolds and Proposition \ref{coneprop} imply the following (see \cite[Corollary 13.3.2]{BG-book}):
\begin{proposition}\label{3-Sask-Einstein}
Suppose $(M, g, (i, j ,k))$ is a $3$-Sasaki manifold. Then $g$ is Einstein with Einstein constant $\dim(M)-1$.
\end{proposition}

Finally, we recall in this section some important identities for $3$-Sasaki manifolds; see e.g. \cite[Proposition 1.2.4]{BG-3Sas} taking our different sign conventions into account.
\begin{proposition}\label{ident_3Sas}
Suppose $(M, g, (i, j ,k))$ is a $3$-Sasaki manifold. Then $i, j, k$ satisfy the following identities and their cyclic permutations in $i, j, k$:
\begin{enumerate}
    \item $i^a\nabla_aj^b=-j^a\nabla_{a}i^b=-k^b$;
    \item $\nabla_bi^c\nabla_a i^b-i_ai^c=-\delta_{a}{}^c$;
    \item $\nabla_bi^c\nabla_aj^b-j_ai^c=-\nabla_bj^c\nabla_ai^b+i_aj^c=\nabla_ak^c$.
\end{enumerate}
\end{proposition}

\subsection{Affine and projective geometry}\label{subsection:affine-projective-geometry}

We briefly review the notions from affine and projective geometry we
need here. For more detailed background of these ideas with a view
toward the present context, see \cite[Sections 2 and 4.1]{GNW}.

A \textit{(torsion-free) affine
 connection} on a manifold $M$ is a (torsion-free) linear connection $\nabla$ on its tangent bundle $TM$.
For an affine manifold $(M,\nabla)$ we denote by
$$R_{ab}{}^c{}_d\in\Gamma(\Wedge^2 T^*M\otimes \End(TM))$$ the
curvature and by $\Ric_{bd}:=R_{cb}{}^c{}_d$ the \emph{Ricci tensor} of
$\nabla$.
Two affine connections $\nabla$ and $\widehat\nabla$ are \textit{projectively equivalent} if they have the same unparametrised geodesics. If both connections are torsion-free, this is well-known to be equivalent to the existence of
a $1$-form $\Upsilon_a\in\Gamma(T^*M)$ such that
\begin{equation}\label{projective_change}
	\widehat{\nabla}_a\xi^b=\nabla_a\xi^b+\Upsilon_a\xi^b+\delta^b{}_a\Upsilon_c\xi^c
\end{equation}
for any vector field $\xi^a\in\Gamma(TM)$.

\begin{definition}
A \textit{projective structure} on a manifold $M$ is an equivalence
class $\mbp$ of projectively equivalent torsion-free affine
connections on $M$, and we call such a structure $(M, \mbp)$ a
\textit{projective manifold}.
\end{definition}

In the following proposition we record some well-known facts about certain important tensors in projective differential geometry (cf. \cite[Section 3]{BEG}):

\begin{proposition}\label{projective_tensors}
Suppose $(M,\mbp)$ is a projective manifold of dimension $n+1\geq
2$. Let $\nabla\in \mbp$ be a connection in the projective class. Then
the curvature $R$ of $\nabla$ admits the decomposition
$$ 
R_{ab}{}^c{}_d= W_{ab}{}^c{}_d +2 \delta^c{}_{[a} \Rho_{b]d}-2\Rho_{[ab]}\delta^c{}_d.
$$ 
Here: $W_{ab}{}^c{}_d$ is the totally trace-free part of $R_{ab}{}^{c}{}_d$, which is independent of the choice of ${\nabla\in\mbp}$, and is called the \textit{projective Weyl curvature} of $(M,\mbp)$; $\Rho_{ab}$ is the projective Schouten tensor of $\nabla$ given by
$$\Rho_{ab}:=\tfrac{1}{n(n+2)}\left[ (n+1) \Ric_{ab}+\Ric_{ba}\right].
$$ Moreover,
$$
\nabla_c W_{ab}{}^c{}_{d}=(n-1)C_{abd},
$$ 
where
$$ 
C_{abc}:= \nabla_a \Rho_{bc}-\nabla_b \Rho_{ac}
$$ 
is called the \textit{projective Cotton tensor} of $\nabla$.
Finally, if $\wh\nabla\in\mbp$ is another connection in the projective class, related to $\nabla$ via $\Upsilon_a$ as in \eqref{projective_change}, then
\begin{align}\label{change_rho}
&\wh \Rho_{ab}=\Rho_{ab}-\nabla_a\Upsilon_b+\Upsilon_a\Upsilon_b,\\
&\wh C_{abc}=C_{abc}+\Upsilon_d W_{ab}{}^d{}_c.
\end{align}
\end{proposition}

\begin{remark} If $n=1$ in Proposition \ref{projective_tensors}, that is, if $\dim(M)=2$, then the symmetries of $W_{ab}{}^c{}_d$ imply that $W = 0$, and $\wh C_{abc}=C_{abc}$.
\end{remark}

Following \cite{BEG}, we fix the following notation for the line bundles on a projective manifold:
\begin{notation} For a projective manifold $(M, \mbp)$ of dimension $n + 1 \geq 2$ we define ${\mathcal K := (\Wedge^{n + 1} TM)^{\otimes 2}}$. Note that $\mathcal K$ enjoys a canonical orientation, and hence we may define preferred oriented roots $\mathcal{E}(w):= \mathcal K^{w / (2 n + 4)}$ for any $w\in\mathbb R$. For any vector bundle $\mathcal V \to M$ we denote $\mathcal V(w):= \mathcal V\otimes \mcE(w)$.
\end{notation}

An immediate consequence of \eqref{projective_change} is:

\begin{proposition}\label{prop_scale}
Suppose $(M, \mbp)$ is a projective manifold and $0\neq w\in \mathbb R$. Then mapping a connection in $\mbp$ to its induced connection on $\mcE(w)$ induces an (affine)
bijection between connections in $\mbp$ and linear connections on $\mcE(w)$. In particular, any trivialisation of $\mcE(w)$, viewed as a  nowhere-vanishing section
$\sigma\in\mcE(w)$, gives rise to a connection $\nabla$ in $\mbp$ characterised by $\nabla_a\sigma=0$.
\end{proposition}

\begin{definition}\label{def_scale} Suppose $(M, \mbp)$ is a projective manifold of dimension $n+1$.
For $0\neq w \in\bbR$, a nowhere-vanishing section of $\mcE(w)$ as well as the corresponding connection $\nabla\in\mbp$ (as in Proposition \ref{prop_scale}) is called a \emph{scale} of $(M,\bf{p})$.
Note that for a scale $\nabla\in \mbp$ one has $\Ric_{[ab]}=\Rho_{[ab]}=0$ and thus $\Ric_{ab}=\Ric_{(ab)}=n\Rho_{(ab)}=n\Rho_{ab}$. Note that if $\nabla, \widehat\nabla\in \mbp$ are both scales then the $\Upsilon$ in (\ref{projective_change}) is exact.
\end{definition}

Any projective manifold $(M,\bf{p})$ of dimension $n + 1$ carries a canonical connection on a vector bundle of rank $n+2$ over $M$, and the curvature of that connection is a (sharp) obstruction to $(M,\bf{p})$ being locally projectively flat \cite[Section 3]{BEG}. Recall that a projective manifold is \emph{locally projectively flat} if it is locally projectively equivalent to the standard projective structure on $\R^{n + 1}$---that is, the projective structure induced by the standard flat affine connection on $\R^{n+1}$, whose geodesics are affine lines. Let us explain this in some more detail: the connection is defined on $\mcT^*:= J^1\mcE(1)\rightarrow M$, the
vector bundle of $1$-jets of sections of $\mcE(1)$.
The natural projection $J^1\mcE(1)\rightarrow \mcE(1)$ induces filtrations on $\mcT$ and its dual $\mcT^*$ given by the dual short exact sequences
\begin{equation}\label{filt_tractor}
    0 \to \mcE(-1) \to \mcT \to TM(-1) \to 0
        \qquad\textrm{and}\qquad
    0 \to T^*M(1) \to \mcT^* \to \mcE(1) \to 0 .
\end{equation}
We shall sometimes index sections of the bundles $\mcT$ and
$\mcT^*$ of a projective manifold by upper respectively lower
Latin capital letters, for example writing $t^A$ for a section of
$\cT$.  Using this notation, we identify the inclusion map $\mcE(-1)
\to \mcT$ above with a section $X^A \in \Gamma(\mcT(1))$ and the
inclusion $T^*M(1) \to \mcT^*$ with a section $Z_A{}^a$.

Note that a splitting of the second sequence of \eqref{filt_tractor} is the same as a linear connection on $\mcE(1)$. Hence, by Proposition \ref{prop_scale}, connections of $\mbp$ are in bijection to splittings of the exact sequences \eqref{filt_tractor}. Given a choice of connection $\nabla\in\mbp$, we write
 \begin{equation}\label{split}
 Y_A:\ce(1) \to \mcT^* \qquad \mbox{and} \qquad W^A{}_a: \mcT^*\to T^*M(1),
 \end{equation}
for the bundle maps of the induced splitting of \eqref{filt_tractor}; thus,
$$
 X^AY_A=1, \qquad  Z_A{}^b W^A{}_a=\delta^b{}_a, \qquad \mbox{and} \qquad Y_AW^A{}_a=0.
$$

Then the following fundamental  result for projective manifolds holds; see \cite[Section 3]{BEG} (and for the last statement of the theorem below \cite[Section 4.1]{GNW}).
\begin{theorem}\label{tractor_thm}
Suppose $(M, \mbp)$ is a projective manifold of dimension $n+1\geq 2$.
\begin{enumerate}
\item[(1)] The projective class $\mbp$ induces on $\mcT^*$ (and on all vector bundles arising from tensor operations of $\mcT^*$ and its dual $\mcT$) a projectively invariant linear connection $\nabla^\mcT$, given by
  $$
\nabla^{\mathcal{T}}_a (\mu_b\, ,\, \si) =(\nabla_a\mu_b+ \Rho_{ab}\si\,,\, \nabla_a\si- \mu_a)
  $$
  in the splitting of $J^1\mcE(1)$ determined by $\nabla\in \mbp$.
\item[(2)] The curvature $R^\mcT$ of $(\mcT, \nabla^\mcT)$ is a $2$-form on $M$ with values in the bundle $\mathfrak{sl}(\mcT)$ of trace-free endomorphisms of $\mcT$, and it vanishes identically if and only if $(M,\bf{p})$ is locally projectively flat. The latter is in turn the case if and only if the projective Weyl curvature vanishes identically for $n\geq 2$, and if and only if the projective Cotton tensor vanishes identically for $n=1$.
\end{enumerate}

Moreover, if $M$ is orientable, then an orientation on $M$ induces a
nowhere-vanishing section $\vol$ of $\Wedge^{n+2} \mcT^*$
that is parallel for $\nabla^\mcT$ (and unique up to a non-zero constant factor).
\end{theorem}

\begin{definition}
Suppose $(M, \mbp)$ is projective manifold. Then the vector bundle $\mcT$ and its dual $\mcT^*$ equipped with the connection $\nabla^\mcT$ of Theorem \ref{tractor_thm} are called the \emph{(standard) tractor} respectively \emph{cotractor bundle} of $(M, \mbp)$, and $\nabla^\mcT$ the \emph{normal tractor (Cartan) connection} of $(M, \mbp)$. Moreover, if  $(M, \mbp)$ is orientable and oriented, we refer to ${\vol \in \Gamma(\Wedge^{n+2} \mcT^*)}$ of Theorem  \ref{tractor_thm} as the \emph{parallel tractor volume form }of $(M,\mbp)$.
\end{definition}

Another important tractor bundle is the following.
\begin{definition} The vector bundle $\mcA:=\mathfrak{sl}(\mcT)$ of trace-free endomorphisms of $\mcT$ (equipped with the connection $\nabla^\mcT$) is called the \emph{adjoint tractor bundle} of $(M,\bf{p})$. It inherits, from \eqref{filt_tractor}, a filtration with a natural projection $\Pi: \mcA\rightarrow TM$.
 \end{definition}

More explicitly, for a choice of connection $\nabla \in \mbp$  (giving rise to a splitting of \eqref{filt_tractor}) the induced filtration of $\mcA$ splits as:
\[
    \mcA\cong TM \oplus {\underbrace{\mathfrak{sl}(TM) \oplus \mcE}_{\End(TM)}} \mathbin{\oplus} {T^* M},
\]
where $\mathfrak{sl}(TM)$ denotes the bundle of trace-free endomorphism of $TM$ and $\mcE=M\times \bbR \rightarrow M$ the trivial bundle.
With respect to such a splitting, any section  $\mathbb A ^A {}_B\in\Gamma(\mcA)$ decomposes as
\begin{equation}\label{equation:adjoint-decomposition}
    \mathbb A^A{}_B = \xi^a W^A{}_a Y_B + \phi^a{}_b W^A{}_a Z_B{}^b -\phi^c{}_c X^A Y_B + \nu_b X^A Z_B{}^b,
\end{equation}
for some $\xi^a \in \Gamma(TM)$, $\phi^a{}_b \in \Gamma(\End TM)$, and $\nu_b \in \Gamma(T^* M)$. We also write such a section more compactly in block matrix notation as:

\[
   \mathbb A^A{}_B= \begin{pmatrix}
        \phi^a{}_b &   \xi^a     \\
         \nu  {}_b & -\phi^c{}_c
    \end{pmatrix}
        \in
    \Gamma
    \begin{pmatrix}
        \End(T  M) & TM    \\
             T^*M  & \mcE
    \end{pmatrix}.
\]
The vector field
\[
    \Pi(\mathbb A)=Z_{A}{}^a X^B \mathbb A^{A}{}_B=\xi^a
\]
is, in particular, independent of the choice of $\nabla$.

For later use we recall that, with respect to a choice
$\nabla\in\mbp$, the curvature $R=R^{\mcT}\in\Gamma(\Wedge^2
T^*M\otimes\mcA)$ of $\nabla^{\mcT}$ may be written as
\begin{equation}\label{tractor_curvature}
R_{ab}{}{}^C{}_{D}=W_{ab}{}^c{}_d W^C{}_c Z_{D}{}^d-C_{abd}Z_{D}{}^d X^C,
\end{equation}
where $W_{ab}{}^c{}_d$ is the projective Weyl curvature and $C_{abd}$
the projective Cotton tensor of $\nabla$, as defined in Proposition
\ref{projective_tensors}.

Parallel sections of $\mcA$ are closely linked to projective
symmetries of $(M,\mbp)$.

\begin{definition}
A vector field $\xi\in\Gamma(TM)$ on a manifold $M$ is called an \emph{affine symmetry} (respectively \emph{projective symmetry}) of a torsion-free affine connection $\nabla$,
if its (local) flow preserves $\nabla$ (respectively the projective class $[\nabla]$ of $\nabla$), equivalently, if $\mathcal L_\xi\nabla=0$ (respectively the trace-free part of $\mathcal L_\xi\nabla$ equals zero).
Here, $\mathcal L_\xi\nabla$ is the \emph{Lie derivative} of $\nabla$ along $\xi$, that is, the $(1,2)$-tensor given by
\begin{equation}\label{Lie_of_nabla}
\mathcal L_\xi\nabla(\eta)=\mathcal L_\xi(\nabla\eta)-\nabla \mathcal L_\xi\eta\quad \textrm{ for } \eta\in\Gamma(TM).
\end{equation}
\end{definition}

We recall the following result from \cite[Section 4.4 and Theorem 4.4]{GNW}:
\begin{theorem}\label{adjoint_tractor_thm}
Suppose $(M,\mbp)$ is a projective manifold of dimension $n+1\geq 3$.
Then,
for any choice $\nabla\in\mbp$, the differential operator $L^{\mcA}:
\Gamma(TM) \rightarrow\Gamma(\mcA)$ given by
\begin{equation}\label{equation:splitting-operator-adjoint}
	L^{\mcA} : \xi^a \mapsto
		\begin{pmatrix}
			  \nabla_b \xi^a - \frac{1}{n +2} \delta^a{}_b \nabla_c \xi^c
			& \xi^a \\
			  -\frac{1}{n + 2} \nabla_b \nabla_c \xi^c{} - \Rho_{bc} \xi^c
			& -\frac{1}{n + 2} \nabla_c \xi^c
		\end{pmatrix} \textrm{.}
\end{equation}
induces a bijection, with inverse $\Pi: \Gamma(\mcA)\rightarrow
\Gamma(TM)$, between vector fields $\xi\in\Gamma(TM)$ that satisfy
\begin{itemize}
\item[(a)] $\xi$ is in the kernel of the projectively invariant differential operator
\begin{align}\label{equation:BGG-operator}
 \mathcal D^{\mcA} :
    \Gamma(TM) \to &\,\Gamma((S^2 T^*M \otimes TM)_{\circ}) \\\nonumber
    \xi^a \mapsto &\,(\nabla_{(b} \nabla_{c)} \xi^a+ \Rho_{(bc)} \xi^a)_\circ,
		\end{align}
where $\,\cdot\,_{\circ}$ denotes the totally trace-free part of $S^2 T^*M \otimes TM$ or the projection thereto, and
\item[(b)] $W_{ab}{}^c{}_d \xi^d= 0$  and $C_{abd} \xi^d=0$,
\end{itemize}
and parallel sections of $\mcA$. Moreover, a solution $\xi$ of $\mathcal D^{\mcA}(\xi)=0$ (satisfying (b)) is a projective symmetry if and only if $W_{d(a}{}^c{}_{b)} \xi^d=0$ (respectively $W_{da}{}^c{}_{b} \xi^d=0$).
\end{theorem}

In view of Theorem \ref{adjoint_tractor_thm} we fix the following notation:

\begin{notation} \label{not_compl_str}

With respect to a connection $\nabla\in\mbp$, of a projective manifold $(M,\mbp)$ of dimension $n+1$, we write a parallel section $\mathbb A^{A}{}_B\in\Gamma(\mcA)$ of the adjoint tractor bundle
as
\begin{equation*}
   \mathbb{A}^A{}_B=
    \begin{pmatrix}
        \nabla_b \xi^a - \varphi(\xi)\delta^a{}_b &   \xi^a     \\
        -\nabla_b\varphi(\xi) - \Rho_{bc} \xi^c & -\varphi(\xi)
    \end{pmatrix} ,
    \qquad
    \textrm{where} \quad \varphi(\xi) := \frac{1}{n+2}\nabla_a\xi^a .
\end{equation*}

\end{notation}

Note that if $\bbI^A{}_B\in\End(\mcT)$ is a complex structure on the
tractor bundle $\mcT$ of a projective manifold, that is,
$\bbI^A{}_B\bbI^B{}_C=-\delta^A{}_C$, then necessarily one has
$\bbI^A{}_B\in\Gamma(\mcA)$. In \cite{GNW} we studied in detail the
geometric implications of the existence of a parallel complex
structure $\bbI\in\Gamma(\mcA)$, and those of a parallel Hermitian
structure, on the tractor bundle of a connected oriented projective
manifold. In this article we will investigate the geometric
implications of parallel tractor hypercomplex and parallel tractor
hyperk\"ahler structures.

\section{Main Results} \label{mainsect}
We are now ready to present the main results of this article.

\subsection{Projective structures induced by $3$-Sasaki structures}
Since a $3$-Sasaki manifold is Einstein, we have:

\begin{theorem}\label{thmA}
Suppose $(M, g, (i, j, k))$ is a $3$-Sasaki manifold of signature $(4p-1, 4q)$, and denote by $\nabla$ the Levi-Civita connection of $g$. Then the tractor bundle $(\mcT, \nabla^{\mcT})$ of the induced projective manifold $(M, [\nabla])$ admits a parallel tractor hyperk\"ahler structure of signature $(4p, 4q)$ given by:
\begin{enumerate}
\item a parallel tractor metric $h \in \Gamma(S^2 \mcT^*)$ of signature $(4p, 4q)$ on $\mcT$, and
\item a triple $(\bbI, \bbJ, \bbK)$ of parallel tractor complex structures $\bbI, \bbJ, \bbK \in \Gamma(\End \mcT)$ on $\mcT$ each Hermitian with respect to $h$ and together satisfying
    \begin{equation}\label{quaternionic}
         \bbI\bbJ\bbK=-\smash{{\id}_{\mcT}} .
    \end{equation}
\end{enumerate}
\end{theorem}

\begin{proof} Suppose $(M, g, (i, j, k))$ is a $3$-Sasaki manifold of signature $(4p-1, 4q)$. Then $g$ is known to be Einstein  (cf.\,Proposition \ref{3-Sask-Einstein}). Hence, Theorem A of \cite{GNW} implies that $g$ induces a parallel tractor metric $h \in \Gamma(S^2 \mcT^*)$ of signature $(4p, 4q)$ on $\mcT$, and, from $i, j$ and $k$, three parallel tractor complex structures $\bbI := L^\mcA(i), \bbJ = L^\mcA(j), \bbK = L^\mcA(k)$ with respect to which $h$ is Hermitian. It remains to show that $\bbI\bbJ\bbK=-\smash{\id_\mcT}$, equivalently that $\bbI\bbJ=\bbK$.
Formula \eqref{equation:splitting-operator-adjoint} for $L^\mcA$ and the fact that $i,j$, and $k$ are Killing fields of $g$ (cf.\ the proof of \cite[Theorem A]{GNW}) imply that
\begin{equation}\label{IJK_form}
	\bbI^A{}_B =
		 	\begin{pmatrix}
		 		\nabla_b i^a & i^a \\
		 		        -i_b & 0
			\end{pmatrix} ,\quad\quad
		\bbJ^A{}_B=
		 	\begin{pmatrix}
		 		\nabla_b j^a & j^a \\
		 		        -j_b & 0
			\end{pmatrix},\quad\textrm{and}\quad \bbK^A{}_B=
		 	\begin{pmatrix}
		 		\nabla_b k^a & k^a \\
		 		        -k_b & 0
			\end{pmatrix},
\end{equation}
in the scale of the Levi-Civita connection of the metric $g$.
Hence, by Definition \ref{Def_3Sas} and Proposition \ref{ident_3Sas} one has

\begin{equation}\label{IJK}
\bbI^A{}_B\bbJ^B{}_C=
		 	\begin{pmatrix}
		 		\nabla_b i^a\nabla_cj^b-i^aj_c & j^b\nabla_b i^a \\
		 		        -i_a\nabla_b j^a & -i_aj^a
			\end{pmatrix} =
		 	\begin{pmatrix}
		 		\nabla_c k^a & k^a \\
		 		        -k_c & 0
			\end{pmatrix}=\bbK^A{}_C. \qedhere
\end{equation}
\end{proof}

Before we investigate the converse of Theorem \ref{thmA}, that is, the geometric implications of a parallel hyperk\"ahler structure on the tractor bundle of a projective manifold, we recall in the following section one more geometric structure; it will appear in results in Section \ref{subsection:local-leaf-space}.

\subsection{Quaternionic contact structures}\label{sec_q_c}
For a manifold $N$ and a distribution $H\subset TN$ we denote by
$\mathcal L: H\times H\rightarrow TN/H$ the \emph{Levi bracket}, which
for two sections $\xi, \eta$ of $H$ is defined as the projection of
their Lie bracket $[\xi, \eta]$ to $TN/H$. Then for any $x\in N$, we
may view $\operatorname{gr}(T_xN)=H_x\oplus T_xN/H_x$ as a Lie algebra,
where the bracket of two elements of $H_x$ is given by $\mathcal
L_x$ and all other brackets are zero.

\newcommand{\tm}{{\widetilde{m}}}

\begin{definition} Suppose $p,q$ are non-negative integers
  and $p+q=\tm =m-1$.
Regard $\mathbb H^{\tilde m}$ as a (right) vector space over $\mathbb H$, and
  denote by $\langle\cdot,\cdot\rangle: \mathbb H^\tm \times \mathbb
  H^\tm \rightarrow \mathbb H $ the standard quaternionic Hermitian
  inner product of signature $(p,q)$, given by
$$\langle x,y\rangle := \sum_{i=1}^p \bar x_i y_i-\sum_{j=p+1}^\tm \bar
  x_jy_j \quad\textrm{ for } x,y\in \mathbb H^{\tilde{m}}.$$ Note that the
  imaginary part $\textrm{Im}(\langle x,y\rangle)\in
  \textrm{Im}(\mathbb H)$ of $\langle x,y\rangle$ is given by
  $\textrm{Im}(\langle x,y\rangle)=\frac{1}{2}(\langle
  x,y\rangle-\langle y,x\rangle)$.  The \emph{quaternionic Heisenberg
    algebra} of signature $(p,q)$ is the $2$-step nilpotent graded Lie
  algebra $$\mathfrak q:=\mathfrak q_{-1}\oplus \mathfrak
  q_{-2}:=\mathbb H^\tm\oplus \textrm{Im}(\mathbb H),$$ where the Lie
  bracket of two elements $(x,a), (y,b) \in \mathbb H^\tm \oplus
  \textrm{Im}(\mathbb H)$ is given by
$$[(x,a), (y,b)]:=(0, \textrm{Im}(\langle x,y\rangle))\in \mathbb H^\tm\oplus \textrm{Im}(\mathbb H).$$
\end{definition}

\begin{definition}\label{quat_contact}
Suppose $N$ is a manifold of dimension $4\tm+3$ ($\tm \geq 1$). Then a
\emph{quaternionic contact structure of signature $(p,q)$} ($p+q=\tm$)
is given by a distribution $H\subset TN$ of rank $4\tm$ such that such
that the graded vector bundle $\textrm{gr}(TN)=H\oplus TN/H$ is
locally trivial as a bundle of Lie algebras $(\textrm{gr}(TN),\mathcal
L)$,  with standard fibre isomorphic to the quaternionic Heisenberg Lie
algebra of signature $(p,q)$.
\end{definition}
Note that the condition on $\textrm{gr}(TN)=H\oplus TN/H$ in
Definition \ref{quat_contact} is equivalent to the existence of a
 quaternionic structure
$Q\subset
\textrm{End}(H) $ on $H$ and a conformal equivalence class $[g]$ of
$Q$-Hermitian (bundle) metrics of signature $(p,q)$ on $H$ (i.e. a
${\operatorname{CSp}(p,q) \times_{\bbZ_2} \Sp(1)}$-structure on $H$)
such that the Levi bracket $\mathcal L: \Wedge^2 H\rightarrow TN/H$
can be locally (for a local trivialisation of the rank-$3$ bundle
$TN/H$) written as $(g(I\,\cdot\,,\,\cdot\,), g(J\,\cdot\,,\,\cdot\,),
g(K\,\cdot\,,\,\cdot\,))$ for a local frame $(I, J, K)$ of $Q$ and a
metric $g\in [g]$.

A compact homogeneous model for quaternionic contact manifolds of
signature $(p,q)$ is the quaternionic projectivisation $$\mathcal
N^{4\tm+3}=\mathcal N^{4(p+q) +3}\subset \mathbb H P^{p+q+1}= \mathbb
H P^{\tm +1}$$ of the cone of null vectors in the standard
quaternionic Hermitian vector space $\mathbb H^{p+1, q+1}$. It
inherits from the quaternionic structure on $\mathbb H P^{\tm+1}$ an
$\Sp(p+1, q+1)$-homogeneous quaternionic contact structure $H\subset T\mathcal N^{4\tm +3}$ of
signature $(p,q)$, where
$H_x$ is the maximal quaternionic subspace of $T_x\mathcal N^{4\tm
  +3}\subset T_x\mathbb H P^{\tm+1}$ for any $x\in \mathcal N^{4\tm
  +3}$. In general, the following is known.
  \begin{theorem} \cite[Section 4.3.3]{CapSlovak} \label{thm_qc}
 There exists an equivalence of categories between quaternionic contact manifolds of
signature $(p,q)$ and regular normal
parabolic geometries modelled on $\Sp(p+1, q+1)/P\cong \mathcal N$,
where $P$ is the stabiliser of a null quaternionic line in $\mathbb H^{p+1, q+1}$.
\end{theorem}
In particular, by Theorem \ref{thm_qc},  the so-called harmonic curvature of the corresponding normal parabolic
geometry is a (sharp) obstruction for the local equivalence of a quaternionic contact manifold to its homogeneous model $\Sp(p+1, q+1)/P$. For $\tm \geq 2$ (that is, $\dim N \geq 11$) the harmonic curvature comprises a single irreducible component called the \emph{curvature} of a quaternionic contact manifold, which can be viewed as a bundle map $\Wedge^2 H \to \End(H)$. For $\tm = 1$ ($\dim N = 7$), the harmonic curvature includes an additional irreducible component, the \emph{torsion} of a quaternionic contact manifold, which can be viewed as a bundle map $H \times TN / H \to TN / H$; in higher dimensions the torsion vanishes. See \cite[Section 4.3.3]{CapSlovak} for more details. Following the literature, we call a $7$-dimensional quaternionic contact structure with vanishing torsion \emph{integrable}.

\begin{remark}
 The (aforementioned) homogeneous model of quaternionic contact manifolds $\mathcal N\subset \mathbb H P^{\tm+1}$ is a real hypersurface in a quaternionic manifold and inherits its quaternionic contact structure from the ambient quaternionic structure. Recall that any generic real hypersurface in a complex manifold inherits a non-degenerate CR-structure of hypersurface type from the ambient complex structure. The analogous statement in the quaternionic setting, however, is not true: If $N$ is a generic real hypersurface of a quaternionic manifold $(M,Q)$, then the maximal $Q$-invariant subbundle $H\subset TN$ given by
\begin{equation}\label{wqc_structure}
    H_x := T_x N\cap I_x (T_xN)\cap  J_x (T_xN)\cap K_x(T_xN) \quad\quad \textrm{ for any } x\in N,
\end{equation}
where $(I_x, J_x, K_x)$ is a hypercomplex structure inducing $Q_x$, is in general merely what is called a \textit{weakly} quaternionic contact structure; see \cite{D2}.
\end{remark}

Quaternionic contact manifolds (in definite signature) were first
studied by Biquard in \cite{BiquardMetriques, BiquardQuaternionic}, where
they arise as geometric structures at infinity of
quaternionic K\"ahler manifolds.  In fact, for $\tm \geq 2$ Biquard showed
in \cite[Theorem D]{BiquardMetriques} that any real-analytic
quaternionic contact structure of definite signature is locally the
(conformal) infinity of a unique so-called asymptotically hyperbolic
quaternionic K\"ahler (AHQK) metric. Duchemin \cite[Theorem 1.4]{D1} showed that same holds for integrable structures with $\tm=1$. Moreover, \cite[Theorem C]{BiquardMetriques} and \cite[Section 5]{D1} (see also \cite[Proposition 4.5.5]{CapSlovak} for the indefinite case) imply that any quaternionic contact manifold $(N,H)$ of signature $(p,q)$ (assumed integrable if $\tm=p+q=1$)
induces on the total of the $S^2$-subbundle $S(Q)\rightarrow N$ (in
$Q\rightarrow N$), of complex structures on $H$ contained in $Q$, a non-degenerate
CR-structure of hypersurface type and (Hermitian) signature $(2p+1,
2q+1)$. The subbundle $S(Q)$ equipped with that
structure is called the \emph{CR-twistor space} of the quaternionic
contact manifold.

Recall also that the classical Fefferman construction \cite{Fefferman,
  BDS} associates to any non-degenerate CR-structure of hypersurface
type a conformal structure on the total space of a certain circle
bundle over the CR-manifold and that conformal structures that arise
via this construction are called \emph{Fefferman conformal
  structures}. Combining the constructions of Biquard and Fefferman
shows that any quaternionic contact manifold of signature $(p,q)$
induces, on the total space of a certain $S^3$-bundle over it, a
conformal structure of signature $(4p+3, 4q+3)$. We call conformal
structures that arise (locally) via this construction from quaternionic contact
manifolds \emph{Biquard--Fefferman conformal structures}.

\subsection{Projective structures with holonomy reduction to \texorpdfstring{$\Sp(p,q)$}{Sp(p, q)}}\label{Sec_Hol_Sp}
We are now ready to prove a result that includes a  converse to Theorem
\ref{thmA}---but also considerably more. We will restrict to dimension at least $11$ and will comment later in Remark \ref{ThmB_dim7} briefly on what happens in the lowest possible dimension, namely $7$, which is special.
\begin{theorem}\label{thmB}
Let $(M, \mbp)$ be a projective manifold of dimension at least $11$
equipped with a parallel hyperk\"ahler structure $(h,
(\bbI,\bbJ,\bbK))$ on its tractor bundle $(\mcT,
\nabla^{\mcT})$. Then, $M$ is of dimension $4m + 3$ and $h$ has
signature $(4p, 4q)$ for some $p, q, \in \{0, 1, \ldots\}$, where
$p+q-1 = m \geq 2$. The  vector fields $i :=
\Pi(\bbI), j := \Pi(\bbJ), k := \Pi(\bbK)$ on $M$ vanish nowhere, and
$M$ is stratified into a disjoint union of submanifolds
$$
    M = M_+ \cup M_0 \cup M_-,
$$
according to the strict sign of $\tau := h(X,X) \in \Gamma(\mcE(2))$, where $X \in \Gamma(\mcT(1))$ is the canonical (weighted) tractor defined in Section \ref{subsection:affine-projective-geometry}. Furthermore, the components $M_{\pm}$ and $M_0$ are each equipped with a geometry canonically determined by $(M, \mbp, h, (\bbI, \bbJ, \bbK))$ as follows.
\begin{enumerate}
	\item The submanifolds $M_{\pm}$ are open and (if non-empty) are respectively equipped with $3$-Sasaki structures $(g_{\pm}, (i, j, k))$ with Ricci curvature $\Ric_\pm = (4 m + 2) g_\pm$; $g_+$ has signature $(4p-1, 4q)$, and $g_-$ has signature $(4q-1, 4p)$. The metrics $g_{\pm}$ are compatible with the projective structure $\mbp$ in the sense that their respective Levi-Civita connections $\nabla^{\pm}$ satisfy $\nabla^{\pm} \in \mbp\vert_{M_{\pm}}$.
	\item The submanifold $M_0$ is (if non-empty) a smooth separating
          hypersurface and is equipped with an oriented (local) Biquard--Fefferman conformal structure $\mbc$ of signature $(4p-1,4q-1)$ as defined in Section \ref{sec_q_c},
               \end{enumerate}
If $h$ is definite, the stratification is trivial, that is,
if $h$ is $\pm$-definite then, respectively, $M = M_\pm$.
\end{theorem}

In Theorem \ref{thmB}(a), and in some places below, we have suppressed, for readability, the notation $\,\cdot\,\vert_{M_{\pm}}$ specifying the restriction of the vector fields $i, j$, and $k$ to $M_{\pm}$.

\begin{proof}[Proof of Theorem \ref{thmB}]
  The existence of a hyperk\"ahler structure on $\mcT$ implies that the rank of $\mcT$ equals $4m+4$, implying that $\dim M =4m+3$, and that $h$ has signature $(4p, 4q)$ for $p+q-1 = m \geq 2$
  (cf.\ \eqref{group_intersection}).
By Proposition 5.23 of \cite{GNW} the vector fields $i, j$, and $k$ vanish nowhere on $M$, and by Theorem 5.1 of \cite{GNW}
the strict sign $\tau=h(X,X)$ induces a decomposition of $M$ as claimed.

\begin{enumerate}
\item Theorem B of \cite{GNW} gives that $M_\pm$ are both open and equipped with Einstein metrics $g_\pm$ with Ricci tensor $(4 m + 2) g_\pm$, whose Levi-Civita connection $\nabla^\pm$ lies in $\mbp\vert_{M_\pm}$. Moreover, we know that $i,j$, and $k$ are Killing fields for $g_\pm$ satisfying $(a)$ and $(b)$ of Definition \ref{Def_Sasaki}, so it remains to show that they also satisfy (c) and (d) of Definition \ref{Def_3Sas}. With respect to $\nabla=\nabla^\pm$ the tractor complex structures $\bbI$ and $\bbJ$ are of the form \eqref{IJK_form}, so the identity $\bbI\bbJ=\bbK$ and \eqref{IJK} imply that $i_aj^a=0$ and that the first identity in Proposition \ref{ident_3Sas}(a), which is equivalent to $[i,j]=-2k$, both hold. Now, since \eqref{quaternionic} implies that $\bbI\bbK=-\bbJ$ and $\bbJ\bbK=\bbI$, one deduces similarly that $i_ak^a=j_ak^a=0$, and that $[j,k]=-2i$ and $[k,i]=-2j$.
\item Theorem B of \cite{GNW} also gives that $M_0$ is a smooth separating hypersurface equipped with a conformal structure $\mbc$ of signature $(4p-1,4q-1)$, and the proof thereof identifies $(\mcT\vert_{M_0}, h\vert_{M_0}, \nabla^\mcT\vert_{M_0})$ with the conformal tractor bundle of $(M_0, \mbc)$ and its normal (conformal) tractor connection. Therefore, the parallel hyperk\"ahler structure $(h, (\bbI, \bbJ, \bbK))$ on the projective tractor bundle $\mcT$ induces a parallel hyperk\"ahler structure on $(\mcT\vert_{M_0}, h\vert_{M_0}, \nabla^\mcT\vert_{M_0})$, and the claim follows from the characterisation of Biquard--Fefferman conformal structures in \cite[Theorem A]{AltQuaternionic}. \qedhere
\end{enumerate}
\end{proof}

\begin{theorem}\label{thmC}
Assume the setting of Theorem \ref{thmB} and that $M_0 \neq \emptyset$ (so $p, q > 0$). Then the projective manifold-with-boundary $(M \setminus M_\mp, \mbp\vert_{M \setminus M_\mp})$ is respectively a projective compactification (of order $2$) of the $3$-Sasaki manifold $(M_\pm, g_\pm, (i, j, k))$, with projective infinity the Biquard--Fefferman conformal structure $(M_0, \pm {\mbc})$.
\end{theorem}
\begin{proof}
This follows immediately from Theorem \ref{thmB} and \cite[Theorem 5.1]{GNW}, which in turn use \cite{CGcompact, CGHjlms, CGH}.
\end{proof}

\begin{remark}\label{ThmB_dim7}
Suppose now $(M, \mbp)$ is a projective manifold of dimension $7$
equipped with a parallel hyperk\"ahler structure $(h, (\bbI,\bbJ,\bbK))$ on its tractor bundle $(\mcT, \nabla^{\mcT})$. Then Theorem \ref{thmB} also holds, except that statement (b) needs to be modified, since the lowest dimension of a Biquard--Fefferman conformal structure (as it fibres over a quaternionic contact manifold) is $10$. More specifically, $h$ is in this case either of definite signature $(8,0)$ respectively $(0,8)$, or signature $(4,4)$. In the definite cases, Theorem \ref{thmB} holds without change: $M_0=\emptyset$ and $M=M_\pm$ admits a $3$-Sasaki structure of definite signature $(7,0)$ respectively $(0,7)$. In the indefinite case, $M_\pm$ admit $3$-Sasaki structures of signature $(3,4)$ and $M_0$ (if non-empty) is a real separating hypersurface, which is canonically equipped with an oriented conformal structure of signature $(3,3)$ with a holonomy reduction of its conformal tractor bundle to $\Sp(1,1)\leq\SO(4,4)$. We will say later, in Remark \ref{ThmD_dim7}, a bit more about the geometric interpretation of this reduction. With this modification, Theorem \ref{thmC} also holds.
\end{remark}

In the following sections will see that the vector fields $i,j$, and $k$ in the setting of Theorem \ref{thmB} span an integrable distribution on $M$ and will investigate the geometric structure on its local leaf space.

\subsection{Projective structures with holonomy reduction to \texorpdfstring{$\SL(m+1, \bbH)$}{SL(m + 1, H)}}\label{Sec_Hol_SL}
We first study the geometric implications of a holonomy reduction of an oriented projective manifold of dimension $4m+3$ to
\[
    \SL(m + 1, \bbH) = \GL(m + 1, \bbH) \cap \SL(4 m + 4, \bbR),
\]
that is, the tractor bundle $(\mcT, \nabla^\mcT)$ admits a parallel
hypercomplex structure
$( \bbI, \bbJ,\bbK )$ such that
$\vol \in \Gamma(\Wedge^{4m+4}\mcT^*)$ and $( \bbI,
\bbJ,\bbK )$ induce the same orientation on $\mcT$. We focus again on the case $m\geq 2$, but comment on the case $m=1$ in Remark \ref{leaf space_dim 7}.

\begin{proposition}\label{ijk-adapted-scales}
Suppose $(M,\mbp)$ is a connected oriented projective manifold of
dimension $4m+3\geq7$ equipped with a parallel hypercomplex structure $(
\bbI, \bbJ,\bbK)$ on its tractor bundle $(\mcT, \nabla^\mcT)$ such
that $\vol \in \Gamma(\Wedge^{4m+4}\mcT^*)$ and $(\bbI, \bbJ,\bbK)$
induce the same orientation on $\mcT$.  Let
$i := \Pi(\bbI)$, $j := \Pi(\bbJ)$, $k := \Pi(\bbK)$ denote the corresponding nowhere-vanishing vector fields on $M$. Then the following hold:
\begin{itemize}
\item[(a)] The distribution $D := \operatorname{span}\{i,j,k\} \subset TM$ has rank $3$ and is integrable.
 \item[(b)] Locally around any point in $M$ there exists a scale $\nabla\in\mbp$ such that $$\nabla_ai^a=\nabla_aj^a=\nabla_ak^a=0.$$ We refer to such a scale as an $(i,j,k)$-adapted scale. In view of \eqref{projective_change} the freedom in the choice of an $(i,j,k)$-adapted scale in $\mbp$ is addition of an exact $1$-form $\Upsilon_a$ satisfying $\Upsilon_ai^a=\Upsilon_aj^a=\Upsilon_ak^a=0$.
\item[(c)] The vector fields $i,j, k$ are affine symmetries of any $(i,j,k)$-adapted scale and so, in particular, are projective symmetries of $(M, \mbp)$.

\end{itemize}
\end{proposition}
\begin{proof}
\leavevmode
\begin{enumerate}
\item[(a)]
To see that $D$ has rank $3$, consider the compact flat homogeneous model of oriented projective manifolds of dimension $4m+3$ given by the sphere $$S^{4m+3}\cong \SL(4m+4,\bbR)/P,$$ viewed as the ray projectivisation of
$\R^{4m+4}$ and equipped with its standard flat projective structure induced by the round metric, where $P$ is the stabiliser of a ray. Since $\SL(m+1,\bbH)$ acts transitively on $S^{4m+3}$ and preserves the parallel tractor $\bbI\wedge \bbJ\wedge \bbK\in\Gamma(\Wedge^3\mcA)$, it follows from \cite{CGH} that the curved orbit decomposition corresponding to the parallel tractor $\bbI\wedge \bbJ\wedge \bbK$ is trivial and that $i\wedge j\wedge k$ is nowhere vanishing,
since the zero set of any normal solution of a BGG operator (corresponding to a parallel tractor) is a union of curved orbits (see in particular, \cite[Section 2.7]{CGH}). Hence, $D$ has rank $3$.
Now let us proof that $D$ is integrable. We fix a scale $\nabla\in\mbp$ and use Notation \ref{not_compl_str}. Then $\bbI \bbJ=\bbK$ implies that
\begin{equation}\label{nabla_ij_id}
j^b\nabla_b i^a-j^a\phi(i)-i^a\phi(j)=k^a \quad\textrm{ and }\quad i^b\nabla_b j^a-i^a\phi(j)-j^a\phi(i)=-k^a,
\end{equation}
after which torsion-freeness of $\nabla$ implies that $[i,j]=\nabla_i j-\nabla_j i=-2k$. Similarly, one deduces from
$\bbI\bbK=-\bbJ$ and $\bbK\bbJ=-\bbI$ that $[k,i]=-2j$ and
$[j,k]=-2i$. Therefore, the distribution $D := \operatorname{span}\{i, j, k\}$
is involutive and hence integrable.
\item[(b)] In Proposition 5.24 of \cite{GNW} we have seen that locally
  around any point in $M$ there exists an $i$-adapted scale, that is,
  a scale $\nabla\in\mbp$ with $\nabla_ai^a=0$ (equivalently,
  $\varphi(i)=0$). According to \eqref{projective_change} the freedom in
  the choice of an $i$-adapted scale is an exact
  $1$-form $\Upsilon_a$
  satisfying $\Upsilon_ai^a=0$. Hence, we need to show that an $i$-adapted
  scale can be modified, according to \eqref{projective_change}, by
  such a $1$-form to a scale that is also adapted to $j$ and $k$. So
  let us fix a local $i$-adapted scale $\nabla$.

First, $\bbI^2=-\smash{\id}_\mcT$ implies that (in the $i$-adapted scale $\nabla$)
\begin{equation}\label{i-adapt_id1}
\Rho_{ab} i^ai^b=1 \quad\textrm{ and }\quad i^a\nabla_ai^b=0.
\end{equation}\label{i-adapt_id2}
Moreover, $\bbI \bbJ=\bbK$ and
$\bbI \bbK=-\bbJ$ imply that
\begin{equation}\label{i-adapt_id2}
\Rho_{ab} i^aj^b=\Rho_{ba} i^aj^b=\varphi(k)\quad\textrm{ and }\quad  \Rho_{ab} i^ak^b= \Rho_{ba} i^ak^b=-\varphi(j),
\end{equation}
and $\bbJ \bbI=-\bbK$,
$\bbK \bbI=\bbJ$ and $\bbJ\bbK=-\bbK\bbJ=\bbI$ that
\begin{equation}\label{i-adapt_id3}
i^a\nabla_a\varphi(j)=-2\varphi(k), \quad i^a\nabla_a\varphi(k)=2\varphi(j), \quad \textrm{and}\quad k^a\nabla_a\varphi(j)=j^a\nabla_a\varphi(k).
\end{equation}
Suppose $\widehat\nabla\in\mbp$ is another $i$-adapted scale which differs from $\nabla$ by $\Upsilon_a$ according to \eqref{projective_change}. Then the identity \eqref{change_rho} shows that:
\begin{equation}
\wh \Rho_{ab}i^b=\Rho_{ab}i^b-(\nabla_a\Upsilon_b)i^b+\Upsilon_a\Upsilon_bi^b=\Rho_{ab}i^b+\Upsilon_b \nabla_ai^b.
\end{equation}
By \eqref{nabla_ij_id}, \eqref{i-adapt_id1} and \eqref{i-adapt_id2}, this implies:
\begin{equation}
\wh \Rho_{ab}i^b j^a=\Rho_{ab}i^bj^a+\Upsilon_b  j^a\nabla_ai^b=\varphi(k)+ \Upsilon_b k^b.
\end{equation}
Analogously, one obtains:
\begin{equation}
\wh \Rho_{ab}i^b k^a=\Rho_{ab}i^bk^a+\Upsilon_b  k^a\nabla_ai^b=-\varphi(j)-\Upsilon_b j^b.
\end{equation}
Therefore, the proof reduces to the question whether locally around any point there exists a smooth function $f$ such that:
\begin{equation}
i\cdot f=0,\quad\quad j\cdot f=-\varphi(j), \quad\quad k\cdot f=-\varphi(k).
\end{equation}
By the Implicit Function Theorem and the Frobenius Theorem, it suffices to show that the rank-$3$ distribution on $M\times\R$ spanned by $\xi:=i$, $\eta:=j+\varphi(j)\frac{\partial}{\partial t}$, and
$\zeta:=k+\varphi(k)\frac{\partial}{\partial t}$ is involutive (here, $t$ denotes the standard coordinate on $\R$). The commutator relations of $i,j$, and $k$ in the proof of (a) and \eqref{i-adapt_id3} yield:
\begin{equation*}
[ \xi, \eta]=-2\zeta\quad\quad
[ \xi,\zeta]= 2\eta,\quad\quad
[\eta,\zeta]=-2\xi,
\end{equation*}
which completes the proof.
\item[(c)] Since, by assumption, $\mcT$ is equipped with a
  $\nabla^\mcT$-parallel hypercomplex structure $(\bbI, \bbJ, \bbK)$,
  the tractor curvature $R_{ab}{}^C{}_D$ is a section of
  $\Wedge^2T^*M\otimes \mathfrak{sl}(\mcT, \bbH)$. Now consider
\[
    R_{AB}{}^C{}_D:=R_{ab}{}^C{}_D Z_{A}{}^aZ_{B}{}^b\in \Gamma(\Wedge^2\mcT^*\otimes \mathfrak{sl}(\mcT, \bbH)) .
\]
With respect to any scale $\nabla \in \mbp$, $R_{AB}{}^C{}_D$ is, per \eqref{tractor_curvature},
$$ R_{AB}{}^C{}_D=W_{ab}{}^c{}_d Z_{A}{}^aZ_{B}{}^b W^C{}_c Z_{D}{}^d-C_{abd} Z_{A}{}^aZ_{B}{}^bX^CZ_{D}{}^d.$$
Proposition \ref{projective_tensors} gives that $W_{[ab}{}^c{}_{d]}=0$ and $C_{[abd]}=0$, and so $R_{[AB}{}^C{}_{D]}=0$. By (the proof of) Lemma 6.5 of \cite{PPS} this implies
\begin{equation}
\bbI^{E}{}_{A}\, \bbI^{F}{}_{B}R_{EF}{}^C{}_D=\bbJ^{E}{}_{A}\, \bbJ^{F}{}_{B}R_{EF}{}^C{}_D=\bbK^{E}{}_{A}\, \bbK^{F}{}_{B}R_{EF}{}^C{}_D=R_{AB}{}^C{}_D.
\end{equation}
Hence $X^A\bbI^{E}{}_{A}R_{EF}{}^C{}_D=X^A\bbJ^{E}{}_{A}R_{EF}{}^C{}_D=X^A\bbK^{E}{}_{A}R_{EF}{}^C{}_D=0$, since $X^AR_{AB}{}^C{}_D=0$. With respect to any scale $\nabla\in\mbp$ this says
that $i^a W_{ab}{}^c{}_{d}=j^a W_{ab}{}^c{}_{d}=k^a W_{ab}{}^c{}_{d}=0$ and $i^a C_{abd}=j^a C_{abd}=k^a C_{abd}=0$. Thus, the claim follows from Proposition 5.27 of \cite{GNW} and Theorem \ref{adjoint_tractor_thm}. \qedhere
\end{enumerate}
\end{proof}

Using Theorem \ref{adjoint_tractor_thm} straightforwardly yields the following identities:

\begin{lemma}\label{identities_IJK=-Id}
Assume the setting of Proposition \ref{ijk-adapted-scales}, and let $\nabla\in\mbp$ be a local $(i,j,k)$-adapted scale. Then $\bbI^2=\bbJ^2=\bbK^2=-\smash{\id}$ and $\bbI\bbJ=-\bbJ\bbI=\bbK$ imply:
\begin{itemize}
\item[(a)] $i^a\nabla_a i^b=0,\quad \Rho_{cd} i^c i^d=1,\quad\Rho_{cd} i^d \nabla_b i^c=0,\quad \nabla_c i^a \nabla_b i^c - i^a \Rho_{bc} i^c=-\delta^a{}_b,$
\\and similarly for $j$ and $k$;
\item[(b)]  $\Rho_{ab} i^a j^b=0,\quad \Rho_{ab} i^a k^b=0, \textrm{ and }\,\Rho_{ab} j^a k^b=0;$
\item[(c)] $i^a\nabla_a j^b=-j^a\nabla_a i^b=k^b,\textrm{ and }\, \Rho_{bc}i^c\nabla_a j^b=-\Rho_{bc}j^c\nabla_a i^b=\Rho_{bc} k^c;$
\item[(d)] $(\nabla_b i^c) \nabla_a j^b-\Rho_{ad} j^d i^c=-(\nabla_b j^c ) \nabla_a i^b+\Rho_{ad} i^d j^c=\nabla_a k^c,$
\end{itemize}
and (c) and (d) also for their cyclic permutations of $i,j$, and $k$.
Moreover, $\nabla^\mcT \bbI=\nabla^\mcT \bbJ=\nabla^\mcT \bbK=0$ implies
\begin{equation}\label{Q_parallel}
\nabla_a \nabla_b i^c=-\Rho_{ab}i^c+\Rho_{bd} i^d\delta^c{}_a,
\end{equation}
and similarly for $j$ and $k$.
\end{lemma}

\begin{theorem}\label{Thm_quaternionic_descent}
Suppose $(M,\mbp)$ is a connected oriented projective manifold of
dimension $4m+3\geq 11$ equipped with a parallel hypercomplex
structure $(\bbI, \bbJ,\bbK)$ on its tractor bundle $(\mcT,
\nabla^\mcT)$ such that $\vol \in \Gamma(\Wedge^{4m+4}\mcT^*)$ and $ (
\bbI, \bbJ,\bbK )$ induce the same orientation on $\mcT$. Then around
any point in $M$ there is an open set $U \subseteq M$ such that the
leaf space $\widetilde M$, of the restriction to $U$ of the
distribution $D := \operatorname{span}\{i, j, k\}$ (in the notation of
\ref{ijk-adapted-scales}(a)), inherits a canonical quaternionic
structure $\widetilde Q \subset \End(T\widetilde{M})$.

Moreover, any $(i,j,k)$-adapted scale $\nabla$ descends to a torsion-free affine connection $\widetilde{\nabla}$ on $\widetilde M$ with respect to which the subbundle $\widetilde Q$ is parallel.
\end{theorem}
\begin{proof} Given $(\bbI,\bbJ,\bbK)$ as in the statement of the theorem,  let $U\subset M$ be a sufficiently small open set
so that $\mbp\vert_U$ admits $(i,j,k)$-adapted scales and so that the
natural projection $\pi: U\rightarrow\widetilde M$ to the leaf space
of $D$ is a submersion. Moreover, fix an $(i,j,k)$-adapted scale
$\nabla\in\mbp\vert_U$. By (a) and (c) of Lemma
\ref{identities_IJK=-Id}, the endmorphisms $\nabla i$, $\nabla j$, and
$\nabla k$ of $TU$ vanish nowhere and map $D$ to $D$. Hence, they
induce endomorphisms $I, J$, and $K$ of $TU/D$.  From (a) and (d) of
Lemma \ref{identities_IJK=-Id} we see that $I,J$, and $K$ are complex
structures on $TU/D$ such that $IJ=K$, which in particular implies
that they span a subbundle $Q\subset \textrm{End}(TU/D)$ of rank
$3$. Note that \eqref{projective_change} implies that $I, J$, and $K$
are independent of the choice of $(i,j,k)$-adapted scale in
$\mbp\vert_U$ and so is $Q\subset \textrm{End}(TU/D)$. We now verify
that $Q$ descends to a subbundle of $\textrm{End}(T\widetilde M)$ by
showing that the flows of $i, j$, and $k$ preserve $Q$. By (c) of
Proposition \ref{ijk-adapted-scales} we know that $i,j$ and $k$ are
affine symmetries of $\nabla$. Thus, (c) of Lemma
\ref{identities_IJK=-Id} (cf.\ also the proof of (a) of Proposition
\ref{ijk-adapted-scales}) implies that one has $\mcL_i\nabla i=\nabla
\mcL_i i=0$ and similarly for $j$ and $k$, and that
\begin{equation*}
\mcL_i\nabla j=\nabla \mcL_i j=-2\nabla k
\end{equation*}
and its cyclic permutations in $i, j$ and $k$ hold. Hence, $Q$ is
preserved by the Lie derivatives $\mcL_i$, $\mcL_j$, and $\mcL_k$, and
so it descends to a subbundle $\widetilde Q\subset
\textrm{End}(T\widetilde M)$. Therefore, $(\widetilde M,\widetilde Q)$
is an almost quaternionic manifold; the quaternionic structure
$\widetilde Q$ is independent of the choice of $(i,j,k)$-adapted scale
$\nabla$ because $Q$ is.

To show that $(\widetilde M,\widetilde Q)$ is moreover quaternionic it
remains to prove that it admits a torsion-free affine connection
$\widetilde\nabla$ preserving $\widetilde Q$. We show that any
$(i,j,k)$-adapted scale induces such a connection. By (a) and (b) of
Lemma \ref{identities_IJK=-Id}, an $(i,j,k)$-adapted scale $\nabla$
induces a decomposition
\begin{equation}\label{decomp_in_scale}
TU=H\oplus D,
\end{equation}
 where $H$ is the common kernel of the nowhere-vanishing $1$-forms
 $\Rho_{ab}i^b$, $\Rho_{ab}j^b$ and $\Rho_{ab}k^b$. Via
 \eqref{decomp_in_scale}, $I,J$, and $K$ may be identified with complex
 structures on $H$, and $Q$ with a subbundle of
 $\textrm{End}(H)$. Moreover, we may define a connection $\nabla^H$ on
 $H$ by $\nabla^H_\eta\xi=p_H(\nabla_\eta\xi)$ for any
 $\xi\in\Gamma(H)$ and $\eta\in\Gamma(TU)$, where $p_H: TU\rightarrow
 H$ denotes the natural projection defined by
 \eqref{decomp_in_scale}. From \eqref{Q_parallel} one sees that
 $\nabla^H$ preserves $I, J$, and $K$ and hence in particular $Q$. Now
 (c) of Proposition \ref{ijk-adapted-scales} implies that $\nabla^H$
 descends to an affine connection $\widetilde\nabla$ on $\widetilde M$,
 with $\widetilde\nabla \widetilde Q=0$; $\widetilde\nabla$ is torsion-free because $\nabla$ is.
\end{proof}

\begin{remark}\label{leaf space_dim 7}
In fact Theorem \ref{Thm_quaternionic_descent} holds also in dimension $7$ (note that the proof did not require dimension $\geq 11$).
In this case, the proof of Theorem \ref{Thm_quaternionic_descent} shows that $(\widetilde M, \widetilde Q)$ is an almost quaternionic $4$-manifold, which, in view of Remark \ref{QK_dim4}, may or may not be called quaternionic
(depending on the chosen convention), and which in any case can be also interpreted as an oriented conformal manifold $(\widetilde M, [\tilde g])$.
\end{remark}

\begin{remark}
In the setting of Theorem \ref{Thm_quaternionic_descent}, suppose that
$\nabla'$ and $\nabla$ are two $(i,j,k)$-adapted scales, which differ
by the exact $1$-form $\Upsilon$ according to
\eqref{projective_change}. Since $\Upsilon$ is exact and annihilates $D=\langle
i,j,k\rangle$, it is invariant under the Lie derivative along $i$, $j$, and $k$.
Hence, there exists a unique $\tilde\Upsilon\in\Gamma(T^*\widetilde M)$ such that
$\pi^*\tilde\Upsilon=\Upsilon$.  Analogously, as in Theorem 5.31
\cite{GNW}, one may verify that $\widetilde\nabla'$ and
$\widetilde\nabla$ are related by $\tilde\Upsilon$ (according to
\eqref{quater_proj_change}).
\end{remark}

\begin{remark}
There is also a converse to Theorem \ref{Thm_quaternionic_descent} associating to any quaternionic manifold an oriented projective manifold with a parallel hypercomplex structure on its tractor bundle,
providing an analogue to the construction described in \cite[Theorem 3.7]{ArmstrongP1} and \cite[Theorem 5.32]{GNW} that associates to any c-projective structure defined as in \cite{CEMN} a projective structure.
We will not carry out this construction here, since it is well-known in other terms; see \cite{PPS, Swann}, as well as \cite{Joyce1, Joyce2}.
\end{remark}

Suppose now that $\widetilde M$ is a manifold of dimension $4m$ equipped with an almost quaternionic structure $\widetilde Q$, which we recall is equivalent to a
$\textrm{GL}(m,\bbH)\times_{\mathbb Z_2}\Sp(1)$-structure on
$\widetilde M$. It is known that  $(\widetilde M, \widetilde Q)$ admits an equivalent description as a regular, normal parabolic geometry modelled on quaternionic projective space
$\textrm{PGL}(m+1,\bbH)/P\cong \bbH P^m$, where $P$ denotes here the stabiliser in $\textrm{PGL}(m+1,\bbH)$ of a point in $\bbH P^m$, see \cite[Section 4.1.8]{CapSlovak}.
Provided that a certain topological obstruction vanishes \cite{Salamon}, a $\textrm{GL}(m,\bbH)\times_{\mathbb Z_2}\Sp(1)$-structure admits a lift to a $\textrm{GL}(m,\bbH)\times\Sp(1)$-structure.
In that case, a choice of such a lift is equivalent to an extension of the corresponding parabolic geometry of type $(\textrm{PGL}(m+1,\bbH), P)$ to a
parabolic geometry of type $(\SL(m+1,\bbH), \widetilde P)$, where  $\widetilde P\leq \SL(m+1,\bbH)$ is the stabiliser of a quaternionic line in $\bbH^{m+1}$.
This implies that for any almost quaternionic manifold $(\widetilde M, \widetilde Q)$ with a choice of a lift to a $\textrm{GL}(m,\bbH)\times\Sp(1)$-structure (which locally always exists)
the normal Cartan connection can be viewed as a canonical linear connection $\nabla^{\widetilde \mcT}$ on the vector bundle $\widetilde \mcT\rightarrow\widetilde M$, which is
the associated bundle to the Cartan bundle with standard fibre the standard representation of $\SL(m+1,\bbH)$; see \cite{CapGoTAMS}. Similarly as for projective structures,
the vector bundle $(\widetilde\mcT, \nabla^{\widetilde \mcT})$ with connection is called the \emph{quaternionic (standard) tractor bundle} of $(\widetilde M, \widetilde Q)$.

\begin{theorem}\label{thm:tractor_descent}
Suppose $(M,\mbp)$ is a connected oriented projective manifold of
dimension $4m+3\geq 7$ equipped with a parallel hypercomplex
structure $(\bbI, \bbJ,\bbK )$ on its tractor bundle
$(\mcT, \nabla^\mcT)$ such that
$\vol \in \Gamma(\Wedge^{4m+4}\mcT^*)$ and $ ( \bbI,
\bbJ,\bbK )$ induce the same orientation on $\mcT$. As in Theorem \ref{Thm_quaternionic_descent} (and Remark \ref{leaf space_dim 7})  let $(\widetilde M, \widetilde Q)$ be any leaf space of the integrable distribution $D\subset TM$.
Write $\pi: M\rightarrow \widetilde M$ for the natural submersion, where $M$ is replaced by a sufficiently small open subset if necessary.
\begin{itemize}
\item[(a)] Declare two elements  $t_x \in\mcT_x$ and $t_{x'}\in\mcT_{x'}$ to be equivalent, if the following holds:
 $\pi(x)=\pi(x')$ and $t_{x'}$ is
the result of parallel-transporting (with respect to $\nabla^\mcT$) $t_x$ along a curve connecting $x$ and $x'$ that is everywhere tangent to $D$. Then this defines an equivalence relation $\sim$ on $\mcT$.
\item[(b)]
The quotient $\widetilde \mcT:=\mcT/\sim$ of $\mcT$ by the equivalence relation $\sim$ as defined in $(a)$ admits a natural structure of a vector bundle over $\widetilde M$ such that $\pi^*\widetilde\mcT\cong \mcT$.
Moreover, the vector space of its sections $\Gamma(\widetilde\mcT)$ may be identified with the subspace of $\Gamma(\mcT)$ consisting of those sections $t\in\Gamma(\mcT)$ that satisfy
$\nabla^\mcT_\xi t=0$ for any section $\xi$ of $D$.
\item[(c)] $\nabla^\mcT$ descends to a linear connection $\nabla^{\widetilde \mcT}$ on  $\widetilde\mcT\rightarrow \widetilde M$.
\item[(d)] The parallel hypercomplex structure $(\bbI, \bbJ, \bbK)$ and the volume form $\vol$ on $\mcT$
descend to a $\nabla^{\wtmcT}$-parallel hypercomplex structure $(\tilde\bbI, \tilde\bbJ, \tilde\bbK)$ and to a $\nabla^{\wtmcT}$-parallel (real) volume form $\tilde\vol\in\Gamma(\Wedge^{4m+4}\wtmcT^*)$ on $\wtmcT$ respectively.
Moreover, $(\bbI, \bbJ, \bbK)$ gives rise to a subbundle $\widetilde {\mathcal S}\subset\wtmcT$ of rank $4$ preserved by $\tilde\bbI, \tilde\bbJ$ and $\tilde\bbK$.
\item[(e)] The vector bundle $(\widetilde\mcT, \nabla^{\widetilde \mcT})$ with connection is the normal quaternionic (standard) tractor bundle of $(\widetilde M, \widetilde Q)$.
\end{itemize}

\end{theorem}
\begin{proof}
~
\begin{enumerate}
\item In the proof of (c) of Proposition \ref{ijk-adapted-scales} we
  have seen that $\xi^aW_{ab}{}^c{}_d=0$ and $\xi^aC_{abd}=0$ for any
  section $\xi$ of $D$.  By the formula for the tractor curvature
  $R^{\mcT}$ in \eqref{tractor_curvature}, this implies that
  $R^{\mcT}(\xi, \,\cdot\,)=0$ for any section $\xi$ of $D$.
  Therefore, the parallel transport along curves with the same
  endpoints that are tangent to $D$ is for all such curves the same,
  as the formula (3.1.13) in \cite{Ballmann}---describing the
  dependence of the parallel-transport of a linear connection on a
  homotopy of curves with fixed endpoints in terms of curvature,
  shows. Hence, $\sim$ is a well-defined equivalence relation on
  $\mcT$.
\item We can proceed analogously to the argument in Section 5.4.5 of
  \cite{GNW}. We equip $\widetilde\mcT$ with the initial topology with
  respect to the natural projection $\tilde p: \widetilde
  \mcT\rightarrow\widetilde M$. Then, choosing local sections of $\pi$
  shows that vector bundle charts of $\mcT$ give rise to local
  trivialisations of $\widetilde\mcT$, and hence $\tilde p: \widetilde
  \mcT\rightarrow\widetilde M$ can naturally be given the structure of
  a smooth vector bundle such that the following diagram commutes:
\begin{center}
\begin{tikzcd}
  \mcT \arrow[r, "\mathrm{proj}"] \arrow[d, "p"]
    & \wtmcT \arrow[d, "\tilde p"] \\
  M \arrow[r, "\pi"]
& \widetilde M \end{tikzcd},
\end{center}
where $\mathrm{proj}:\mcT \rightarrow \wtmcT=\mcT/\sim$ denotes the natural projection. The universal property of the pullback bundle and the fact that
$\mathrm{proj}\vert_{\mcT_x}: \mcT_x\rightarrow \wtmcT_{\pi(x)}$ is an isomorphism for all $x\in M$ immediately imply that $\pi^*{\wtmcT}\cong \mcT$.
Moreover, it follows that for any section $\tilde t\in\Gamma(\wtmcT)$ the composition $\tilde t\circ\pi$ can be naturally interpreted as a section of $\mcT$ and evidently the sections $t$ of $\mcT$ that arise in this way from sections $\wtmcT$ are precisely those that satisfy $\nabla^{\mcT}_\xi t=0$ for any section $\xi$ of $D$.
\item Suppose $\tilde\xi\in\Gamma(T\widetilde M)$ is any vector field on $\widetilde M$ and choose a lift of it to a vector field $\xi\in\Gamma(TM)$ on $M$ and let $\tilde t$ be section of $\widetilde\mcT$, which we view as a section of $\mcT$ satisfying $(\nabla^\mcT \tilde t)\vert_D=0$. Note that $\nabla_\xi^{\mcT} \tilde t$ is independent of the choice of lift $\xi$, since $(\nabla^\mcT \tilde t)\vert_D=0$. Moreover, for any section $\eta\in\Gamma(D)$ we have $[\xi, \eta]\in \Gamma(D)$, since $\xi$ is projectable as a lift of $\tilde\xi$. Thus, for any $\eta\in\Gamma(D)$ we have
$$\nabla_\eta^\mcT\nabla_\xi^\mcT\tilde t=R^{\mcT}(\eta, \xi)(\tilde t)+\nabla_\xi^\mcT\nabla_\eta^\mcT\tilde t+\nabla_{[\eta, \xi]}^\mcT\tilde t=0,$$
since $R^{\mcT}$ vanishes upon insertion of sections of $D$ as we have seen in the proof of (a) and $(\nabla^\mcT \tilde t)\vert_D=0$. Therefore, $\nabla^\mcT_\xi\tilde t$ defines a section of $\widetilde\mcT$, which is independent of the choice of lift $\xi$. Denoting this section by $\nabla^{\wtmcT}_{\tilde\xi}\tilde t$, one verifies straightforwardly that $\nabla^{\wtmcT}: \Gamma(\wtmcT)\rightarrow \Gamma(T^*\widetilde M\otimes \wtmcT)$ defines a linear connection on
$\wtmcT$.
\item The first statement follows directly from the definition of $\nabla^{\wtmcT}$ and $\nabla^\mcT$-parallelism of $(\bbI,\bbJ, \bbK)$ and $\vol$. For the second statement denote by $\mathcal S\subset \mcT$ the subbundle given by the tensor product of $\mcE(-1)$ with the subbundle of $\mcT(1)$ generated by $X$, $\bbI X$, $\bbJ X$, and $\bbK X$, where $X$ is the canonical (weighted) tractor $X \in \Gamma(\mcT(1))$ defined by
\eqref{filt_tractor}. By construction, $\mathcal S\subset \mcT$ is a subbundle of rank $4$, which is invariant under the action of $\bbI$, $\bbJ$ and $\bbK$. We claim it descends to a rank-$4$ subbundle
$\widetilde{\mathcal S}\subset\wtmcT$ invariant under $\tilde\bbI, \tilde\bbJ$ and $\tilde\bbK$. Hence, we have to verify that $\nabla^{\mcT}_\xi$ preserves $\Gamma(\mathcal S)$ for any section $\xi$ of $D$.
To do so, choose a (locally defined) $(i,j,k)$-adapted scale. With respect to such a scale, we have, by \eqref{equation:adjoint-decomposition}, that $\bbI_{B}{}^AX^A=i^aW^A{}_a$,  $\bbJ_{B}{}^AX^A=j^aW^A{}_a$ and  $\bbK_{B}{}^AX^A=k^aW^A{}_a$. Since  $\nabla_aX^A=W^A{}_a$, this implies that $\xi^a\nabla^{\mcT}_a X^A\in\Gamma(\mathcal S)$ for any $\xi\in\Gamma(D)$. Moreover, contracting $\nabla^{\mcT}_a\bbI_{B}{}^A X^B=\bbI_{B}{}^A \nabla^{\mcT}_aX^B$
with either $i^a$, $j^a$ or $k^a$ gives $-X^A$, $\bbK_{B}{}^A X^B$ and  $-\bbJ_{B}{}^A X^B$ respectively. Hence, also $\xi^a\nabla^{\mcT}_a \bbI_{B}{}^A X^B\in\Gamma(\mathcal S)$ for any $\xi\in\Gamma(D)$ and the analogous statement for $\nabla^{\mcT}_a \bbJ_{B}{}^A X^B$ and $\nabla^{\mcT}_a \bbK_{B}{}^A X^B$ follows similarly.
\item Set $\tilde\mfg:=\mathfrak{sl}(m+1,\bbH)$ and let $\widetilde P\leq \SL(m+1,\bbH)$ be the stabiliser of a quaternionic line in $\bbH^{m+1}$. By the characterisation of tractor bundles and tractor connections in \cite{CapGoTAMS}, we know that the normal (standard) quaternionic tractor bundle with connection of $(\widetilde M, \widetilde Q)$
is (up to natural equivalence) the unique vector bundle with structure group $\wtP$ and standard fibre the $(\tilde \mfg, \wtP)$-module $\bbH^{m+1}$ and
normal non-degenerate $\tilde\mfg$-connection. Hence, (e) can be proved by verifying all these properties for our pair $(\wtmcT, \nabla^{\wtmcT})$.
The proof is completely analogous to the proof of Proposition 5.35 of \cite{GNW}, so we will only sketch it here. That $(\wtmcT, \nabla^{\wtmcT})$ is a vector bundle with structure group $\wtP$ and standard fibre $\bbH^{m+1}$
follows immediately from the existence of the hypercomplex structure $(\tilde\bbI, \tilde\bbJ, \tilde\bbK)$, the volume form $\tilde{\vol}$,
and the filtration $\tilde\mcS\subset \wtmcT$, where $\tilde\mcS$ corresponds to the quaternionic line in $\bbH^{m+1}$ stabilised by $\wtP$. Since $\nabla^{\mcT}$ has holonomy contained in
$\SL(m+1,\bbH)<\SL(4m+4,\bbR)$, the construction of $\nabla^{\wtmcT}$ implies that it is a $\tilde\mfg$-connection on $\wtmcT$. Moreover, normality of $\nabla^{\mcT}$ implies
normality of
$\nabla^{\wtmcT}$, since the curvature $R^{\wtmcT}$ of $\wtmcT$ is the descent of $R^{\mcT}$. Finally, an analogue of the proof of Claim 3 of \cite[Proposition 5.35]{GNW} implies the non-degeneracy of $\nabla^{\wtmcT}$, since the vector bundle map
$$\tilde\Psi: T\wtM\rightarrow \Hom(\tilde\mcS, \wtmcT/\tilde\mcS),$$
induced by  the composition of $\nabla^{\wtmcT}$ with the natural projection $\Pi^{\wtmcT}: \wtmcT\rightarrow \wtmcT/\tilde\mcS$,
defines an isomorphism between $T\wtM$ and the vector subbundle $\Hom_{(\tilde\bbI,\tilde\bbJ, \tilde\bbK)}(\tilde\mcS, \wtmcT/\tilde\mcS)$ of $\Hom(\tilde\mcS, \wtmcT/\tilde\mcS)$ of endomorphisms commuting with
$\tilde\bbI,\tilde\bbJ$ and $\tilde\bbK$. \qedhere
\end{enumerate}
\end{proof}
\begin{remark}\label{R1}
In the proof of Theorem \ref{Thm_quaternionic_descent} we have seen
that the complex structures $I,J$, and $K$ on $TU/D$ are independent
of the choice of $(i,j,k)$-adapted scale. Alternatively, this can be
also seen more directly as follows: Let $\mathcal S\subset \mcT$ be as
in the proof of (d) of Theorem \ref{thm:tractor_descent}.  Since
$\bbI, \bbJ, \bbK$ preserve $\mcS\subset \mcT$, they induce complex
structures on $\mcT/\mcS\cong TM/D\otimes \mcE(-1)$ satisfying the
quaternionic relations and thus on also on $TM/D$. The latter coincide
with $I,J, K$.
\end{remark}

\subsection{The local leaf space of the distribution spanned by $i,j$, and $k$ in Theorem \ref{thmB}}\label{subsection:local-leaf-space}
We return now to the setting of Theorem \ref{thmB}: Let $(M, \mbp)$ be
a projective manifold of dimension at least $11$ equipped with a
parallel hyperk\"ahler structure $(h, (\bbI,\bbJ,\bbK))$ on its
tractor bundle $(\mcT, \nabla^{\mcT})$.  Let $(\widetilde M,
\widetilde Q)$ be the local leaf space of the $3$-dimensional
foliation $D$ determined by $(i, j, k)$ as in Theorem
\ref{Thm_quaternionic_descent} and write $\pi: M\rightarrow \widetilde
M$ for the natural submersive projection (where $M$ is replaced, if
necessary, by a sufficiently small open subset).

By Theorem \ref{thm:tractor_descent}, the $(\bbI,\bbJ,\bbK)$-Hermitian bundle metric $h$ of signature $(4p,4q)$ on $\mcT$ descends to a $\nabla^{\wtmcT}$-parallel bundle metric $\tilde h$ of signature $(4p,4q)$ on $\wtmcT$, which is Hermitian with respect to the $\nabla^{\wtmcT}$-parallel hypercomplex structure $(\tilde\bbI, \tilde\bbJ, \tilde\bbK)$ on $\wtmcT$. Hence, $\tilde h$ defines a holonomy
reduction of the quaternionic standard tractor bundle to $$\Sp(p,q)\leq \SL(m+1,\bbH).$$
By the theory of holonomy reductions of Cartan connections \cite{CGH}, this implies in particular that $\wtM$ admits a decomposition into initial submanifolds, where each such submanifold corresponds
to an $\Sp(p,q)$-orbit in $\SL(m+1,\bbH)/\wtP\cong \bbH P^m$ of the same dimension. Note that there are three $\Sp(p,q)$-orbits on $\bbH P^m$ corresponding to the signature of a quaternionic
line in $\bbH^{m+1}$ being positive, negative, or null. The first two orbits are open, whereas the last defines a smooth embedded (real) hypersurface in $\bbH P^m$ separating the two open orbits. Therefore,  \cite{CGH}
implies that $\wtM$ decomposes analogously as
\begin{equation}\label{decom_leaf_space}
\wtM=\wtM_+\cup \wtM_0\cup \wtM_-,
\end{equation}
where $\wtM_\pm\subset \wtM$ are open submanifolds (if non-empty) and $\wtM_0\subset \wtM$ is a smooth embedded (real) hypersurface. To describe these submanifolds more explicitly
let us write $S^2_+\wtmcT^*\subset S^2\wtmcT^*$ for the subbundle of
$(\tilde\bbI, \tilde\bbJ, \tilde\bbK)$-Hermitian forms on $\wtmcT$. Then, restricting a Hermitian form on $\wtmcT$ to the $(\tilde\bbI, \tilde\bbJ, \tilde\bbK)$-invariant subbundle
$\widetilde\mcS$ defines a surjection $$\Pi^{S^2_+\wtmcT^*}: S^2_+\wtmcT^*\rightarrow S^2_+\widetilde\mcS^*$$ from $(\tilde\bbI, \tilde\bbJ, \tilde\bbK)$-Hermitian forms on $\wtmcT$ to the line bundle $S^2_+\widetilde\mcS^*$ of  $(\tilde\bbI, \tilde\bbJ, \tilde\bbK)$-Hermitian forms on $\widetilde\mcS$.
Then, $\wtM_\pm$ and $\wtM_0$ are the subsets of $\wtM$, where $\tilde\tau=\Pi^{S^2_+\wtmcT^*}(\tilde h)\in \Gamma(S^2_+\widetilde\mcS^*)$ is positive, negative or zero respectively.
By definition of $\tilde h$, we must also have $\pi^*\tilde\tau=\tau=h(X,X)$ and hence $$\wtM_\pm=\pi(M_\pm)\, \textrm{ and }\, \wtM_0=\pi(M_0).$$ In particular, $\wtM_0$ separates $\wtM_\pm$ as $M_0$ separates $M_\pm$.

\begin{theorem} \label{thmD}
Assume the setting of Theorem \ref{thmB} and let $(\widetilde M, \widetilde Q)$ be the local leaf space of the $3$-dimensional foliation $D$ determined by $(i, j, k)$  as in Theorem \ref{Thm_quaternionic_descent}. We write
$\pi: M\rightarrow \widetilde M$ for the natural submersion (where $M$ is replaced, if necessary, by a sufficiently small open subset) and let $\wt M_\pm=\pi(M_\pm)$ and $\wt M_0=\pi(M_0)$ as in \eqref{decom_leaf_space}.

Then we have:
\begin{enumerate}
    \item The submanifolds $\wt M_{\pm}$ are open and (if non-empty) are respectively equipped with quaternionic K\"ahler structures $(\widetilde Q, \tilde g_{\pm})$ of signature $(4(p-1), 4q)$ and $(4(q-1), 4p)$.
    The metrics $g_\pm$ underlying the quaternionic K\"ahler structures have Ricci tensor $\widetilde{\Ric}_\pm = (4 m + 8)\tilde g_\pm$.
    \item The submanifold $\wt M_0$ (if non-empty) is a smooth real
      hypersurface in $(\widetilde M, \widetilde Q)$, and the maximal
      $\widetilde Q$-invariant subbundle $\widetilde H_0\subset
      T\widetilde M_0$, as defined in \eqref{wqc_structure}, defines a
      quaternionic contact structure of signature $(p-1, q-1)$ on $\widetilde M_0$, with
      quaternionic structure on $\widetilde H_0$ equal to $\tilde
      Q\vert_{\widetilde H_0}$. Moreover, the conformal structure
      associated to $(\widetilde M_0, \widetilde H_0)$, via the
      Biquard--Fefferman construction (as explained in Section
      \ref{sec_q_c}), equals $(M_0, \pm\mbc)$.
   \end{enumerate}
\end{theorem}
\begin{proof}
~
\begin{enumerate}
    \item It is well-known that the local leaf space of the canonical rank-$3$ distribution of a $3$-Sasaki structure admits a natural quaternionic K\"ahler structure; see \cite[Theorem 13.3.13]{BG-book}. Hence, (a) follows directly from Theorem \ref{thmB}. For completeness and convenience of the reader, let us nevertheless give a proof of (a). Let $(g_\pm, (i,j,k))$ be the $3$-Sasaki structure on $M_\pm$ of Theorem \ref{thmB}, and denote by $\nabla^\pm\in \mbp\vert_{M_\pm}$ the Levi-Civita connection of $g_\pm$. Since $i,j$, and $k$ are Killing fields of $g_\pm$ (again we suppress the restriction notation $\,\cdot\,\vert_{M_{\pm}}$), the metric  $g_\pm$ descends to a metric $\tilde g_\pm$ on $\widetilde M_\pm$ characterised by
        \begin{equation*}
            \tilde g_{\pm}(T\pi \cdot \xi, T\pi \cdot \eta)=g(\xi, \eta) \quad\quad \forall\, \xi,\eta\in D^\perp,
        \end{equation*}
        where $D^\perp$ denotes the orthogonal complement of $D$ with respect to $g_\pm$. The fact that $i,j$, and $k$ are Killing fields of $g_\pm$ also implies that $\nabla^\pm i$, $\nabla^\pm j$ and $\nabla^\pm k$ are $g_\pm$-skew. 
        It follows from the construction of $\widetilde Q$ and $\tilde g_\pm$ that $\tilde g_\pm$ is $\widetilde Q$-Hermitian. Moreover, by Theorem \ref{Thm_quaternionic_descent}, $\nabla^\pm$ descends to a torsion-free affine connection $\widetilde\nabla^\pm$ on $\widetilde M_\pm$ with $\widetilde\nabla^\pm \widetilde Q=0$, which has to be the Levi-Civita connection of $\widetilde g_\pm$ by the definition of $\tilde g_\pm$ and the fact that $\nabla^\pm g_\pm=0$. Thus, $(\widetilde M_\pm, \widetilde Q, \tilde g_\pm)$ are quaternionic K\"ahler structures of signature $(4(p-1), 4q)$ and $(4(q-1), 4p)$, respectively. By Theorem \ref{thmB}, we have $\Ric_\pm=(4m+2) g_\pm$, from which one deduces by computation that $\widetilde{\Ric}_\pm=(4m+8) \tilde g_\pm$; see \cite[Theorem 13.3.13]{BG-book}.

    \item We have already seen that $\widetilde M_0$ is a smooth (real) hypersurface separating $\widetilde M_+$ and $\widetilde M_-$ in the quaternionic manifold $(\widetilde M, \widetilde Q)$.
        We will show that the maximal quaternionic subbundle $\widetilde
              H_0\subset T \widetilde M_0$ defines a quaternionic contact structure on $\widetilde M_0$ by
              identifying it with the quaternionic contact structure over which
              the Biquard--Fefferman conformal structure $(M_0,
              \mbc)$ fibres, by dint of its definition. To see
              this we proceed along the same lines as the proof of Theorem D in
              \cite{GNW}. As we have seen there, $(\mcT\vert_{M_0},
              h_0)$, where $h_0 := h\vert_{M_0}$, is the conformal tractor bundle of $(M_0, \bf{c})$
              which is filtered as $$\mcT\vert_{M_0} \supset \mcT^0
              \supset \mcT^1,$$ where $\mcT^1$ is the restriction of
              the line bundle $X\mcE(-1)\subset\mcT$ to $M_0$ and
              $\mcT^0:=(\mcT^1)^\perp$ its orthogonal subspace with
              respect to $h_0$. The containment $\mcT^1 \supset \mcT^0$ follows from the fact that the canonical weighted tractor $X \in \Gamma(\mcT(1))$ is null on $M_0$. In particular,
              in the notation of \cite[Section 2]{BEG} for line bundles $\mcE[w]$ on
              conformal manifolds one has $\mcT^1\cong \mcE[-1]$ and
              $\mcT^0/\mcT^1\cong TM_0[-1]$, and the conformal
              structure $\mbc$ on $M_0$ is given by the
              non-degenerate bilinear form $\mbg\in\Gamma(S^2
              T^*M_0[2])$ which $h_0$ induces on
              $\mcT^0/\mcT^1$. The proof of Theorem D of
              \cite{GNW} shows that the respective projecting parts $i,j$, and $k$ of
              $\bbI, \bbJ$, and $\bbK$ restrict on $M_0$ to conformal Killing
              fields that are null.
              Moreover, the fact that $h$ is Hermitian with respect to the hypercomplex structure $(\bbI, \bbJ,\bbK)$, implies also that
              $i,j,k$ are orthogonal to each other: for $i$ and $j$ we have
              $$\mbg(i,j)=h(\bbI X,\bbJ X)=-h(X,\bbK X)=h(\bbK X,
              X)=h(X,\bbK X),$$ which thus has to vanish, and
              similarly for $(k,i)$ and $(j,k)$.  Therefore,
              $\mbg(i,\,\cdot\,)$, $\mbg(j,\,\cdot\,)$, and
              $\mbg(k,\,\cdot\,)$ descend to nowhere-vanishing
              linearly independent (weighted) $1$-forms on $TM_0/D_0$,
              where $D_0 := D\vert_{M_0}$ is the distribution spanned
              by $i,j$, and $k$ on $M_0$. Let us denote the common
              kernel of these $1$-forms by $H_0\subset TM_0/D_0$. Note
              that, in the notation of Remark \ref{R1}, $H_0$ can be
              identified, up to a twist with a line bundle, with the
              bundle $\mcS_0^\perp/\mcS_0\subset
              TM_0/D_0\otimes\mcE[-1]$, where $\mcS_0 :=
              \mcS\vert_{M_0}$ is the restriction of $\mcS\subset\mcT$
              (as defined in the proof of (d) of Proposition
              \ref{thm:tractor_descent}) to $M_0$.  Since $(h_0,
              (\bbI_0,\bbJ_0,\bbK_0))$ is a hyperk\"ahler structure on
              $\mcT\vert_{M_0}$ (here $\bbI_0 := \bbI \vert_{M_0}$ and
              analogously for $\bbJ_0$ and $\bbK_0$), the complex
              structures $\bbI_0, \bbJ_0, \bbK_0$ preserve $\mcS_0$
              and its orthogonal complement. Hence, they
              (respectively) induce complex structures $I_0, J_0, K_0$
              on $H_0$, and by construction they are the restrictions
              of the complex structures $I,J$, and $K$ induced on
              $TM/D$; see Remark \ref{R1}.  By \cite[Section 4.2]{AltQuaternionic} and \cite{AltTwistor},
              $H_0$ and the quaternionic structure
              $Q_0=\operatorname{span}\{I_0,J_0,K_0\}$ descend via
              $\pi$ (the projection to the local leaf space of $D_0$)
              to the quaternionic contact structure $(\widetilde H_0',
              \widetilde Q_0')$ over which the Biquard--Fefferman
              conformal structure $(M_0, \mbc)$ fibres. By the
              definition of $\widetilde Q$, the subbundle $\widetilde
              H_0'\subset T\widetilde M_0$ must be $\widetilde
              Q$-invariant and $\widetilde Q_0'$ must be the
              restriction of $\widetilde Q$ to $\widetilde
              H_0'$. Since $\widetilde H_0$ is by definition (see
              statement (b)) the maximal $\widetilde Q$-invariant
              subbundle of $T\widetilde M_0$, we conclude
              $\widetilde H_0'\subset \widetilde
              H_0$ and hence for dimensional reasons $\widetilde
              H_0'=\widetilde H_0$. \qedhere
\end{enumerate}
\end{proof}
We finish with a few remarks.

\begin{remark}\label{ThmD_dim7}
Let us briefly comment on the analogue of Theorem \ref{thmD} in the
$7$-dimensional case. Suppose $(M, \mbp)$ is a projective manifold of
dimension $7$ equipped with a parallel hyperk\"ahler structure $(h,
(\bbI,\bbJ,\bbK))$ on its tractor bundle $(\mcT, \nabla^{\mcT})$. Then
$(\widetilde M, \widetilde Q)$ is $4$-dimensional, which may be also
viewed as an oriented conformal $4$-manifold, as explained in Remark
\ref{leaf space_dim 7}. By Theorem 13.3.13 of \cite{BG-book}, it follows that on
$\widetilde M_\pm$ one gets quaternionic K\"ahler $4$-manifolds as
defined in Remark \ref{QK_dim4}.  If $\widetilde M_0\neq \emptyset$,
which implies $p=q=1$, then $\widetilde M_0$ is a smooth hypersurface
in $\widetilde M$, separating $\widetilde M_+$ and $\widetilde M_-$,
with a conformal structure of definite signature.  The latter is
induced from the conformal structure on $M_0$ with holonomy reduction
to $\Sp(1,1)\cong \textrm{Spin}(4,1)$ (cf. Remark
\ref{ThmB_dim7}). Recall in this context that LeBrun has shown in
\cite{LeBrun} that any real-analytic $3$-dimensional conformal
manifold is naturally the conformal infinity of a germ-unique
real-analytic $4$-dimensional Riemannian manifold that is self-dual
and Einstein.
\end{remark}

\begin{remark}
Note that in part (a) of Theorem \ref{thmD}, the underlying quaternionic K\"ahler structures  $(\widetilde Q, \tilde g_{\pm})$ on $\wt M_\pm$ are never hyperk\"ahler, since hyperk\"ahler structures are Ricci-flat.
\end{remark}

\begin{remark}
In principle one can adapt the methods of \cite[\S6]{GNW} to use
quaternionic contact structures to construct examples of projective
structures with holonomy reductions to $\Sp(p, q)$, including ones
with a non-empty hypersurface curved orbit. In that reference, for a
CR structure to produce (along with some other choices) a smooth
projective structure with special unitary holonomy reduction, it was
necessary that the Fefferman--Graham ambient obstruction tensor
$\mathcal{O}$ of the induced conformal structure vanishes. For
Biquard--Fefferman conformal structures, however, $\mathcal{O}$
vanishes automatically \cite[Corollary 1.1(4)]{LeistnerLischewski}.
\end{remark}

\begin{remark}
One could weaken the hypothesis of Theorem \ref{thmB}  and instead ask for the curved orbit decomposition
determined by a parallel tractor quaternionic K\"ahler structure (that
is, a holonomy reduction of the normal projective tractor connection
to $\Sp(p, q) \times_{\bbZ_2} \Sp(1)$) rather than a parallel hyperk\"ahler
structure (a reduction to $\Sp(p, q)$). But \cite[Theorem
  4.5.1]{ArmstrongThesis} or \cite[Proposition 4.15]{ArmstrongP1}
implies if a projective structure $(M, \mbp)$ admits a parallel
tractor quaternionic K\"ahler structure, then locally around each
point (or globally, if $M$ is simply connected) it in fact admits a
parallel hyperk\"ahler structure, so (locally) this weaker reduction
has the same geometric consequences. This result mirrors the fact explained
in Section 5.5 of \cite{GNW} that a parallel tractor K\"ahler
structure (a reduction to $\U(p', q')$) (locally) entails a parallel
tractor Calabi--Yau structure (a reduction to $\SU(p', q')$).
\end{remark}

\end{document}